\documentclass [11pt,oneside,a4paper,mathscr]{amsart}

\usepackage{amscd,amsmath,amssymb,euscript}
\usepackage[frame,cmtip,curve,arrow,matrix,line,graph]{xy}
\usepackage{tikz-cd}
\usepackage{adjustbox}
\usetikzlibrary{matrix}

\usepackage{hyperref}
\usepackage{stmaryrd}

\title[Hall algebras for quivers with potential]{Categorical and K-theoretic Hall algebras for quivers with potential}
\author{Tudor P\u adurariu}
\address{Department of Mathematics, Columbia University, 
2990 Broadway, New York, NY 10027}
\email{tgp2109@columbia.edu}

\date{}

\newtheorem{thm}{Theorem}[section]
\newtheorem{cor}[thm]{Corollary}
\newtheorem{prop}[thm]{Proposition}

\theoremstyle{definition}
\newtheorem{defn}[thm]{Definition}

\newtheorem{thm*}[thm]{Theorem$^*$}

\newtheorem{conjecture}[thm]{Conjecture}

\newcommand{\comment}[1]{}

\renewcommand{\leq}{\leqslant}
\renewcommand{\geq}{\geqslant}

\newcommand{\F}{\mathcal{F}}

\newcommand{\A}{\mathcal A}

\newcommand{\Z}{\mathbb{Z}}

\newcommand{\X}{\mathcal{X}}

\newcommand{\C}{\mathbb{C}}

\newcommand{\ee}{\underline{e}}
\newcommand{\dd}{\underline{d}}

\begin{document}
\maketitle

\begin{abstract}
   Given a quiver with potential $(Q,W)$, Kontsevich--Soibelman constructed a Hall algebra on the critical cohomology of the stack of representations of $(Q,W)$. Special cases of this construction are related to work of Nakajima, Varagnolo, Schiffmann--Vasserot, Maulik--Okounkov, Yang--Zhao etc. about geometric constructions of Yangians and their representations; indeed, given a quiver $Q$, there exists an associated pair $\left(\widetilde{Q},\widetilde{W}\right)$ whose CoHA is conjecturally the positive half of the Maulik--Okounkov Yangian $Y_{\text{MO}}(\mathfrak{g}_Q)$. 

For a quiver with potential $(Q,W)$, we follow a suggestion of Kontsevich--Soibelman and study a categorification of the above algebra constructed using categories of singularities. Its Grothendieck group is a K-theoretic Hall algebra (KHA) for quivers with potential. We construct representations using framed quivers and we prove a wall-crossing theorem for KHAs.
We expect the KHA for $\left(\widetilde{Q},\widetilde{W}\right)$ to recover the positive part of quantum affine algebra $U_q\left(\widehat{\mathfrak{g}_Q}\right)$ defined by Okounkov--Smirnov.
\end{abstract}

\section{Introduction}

\subsection{Quivers with potential.}
Let $Y$ be a Calabi-Yau $3$-fold and let $\beta\in H^{4}(Y,\mathbb{Z})\oplus H^{6}(Y,\mathbb{Z})$.
The Donaldson--Thomas (DT) invariants of $Y$ are virtual counts of curves on $Y$ of support $\beta$ and they are defined using the geometry of 
the moduli stack $\mathfrak{M}_\beta$ of sheaves with compact support $\beta$.
One can define DT invariants for quivers with potential $(Q,W)$ using the vanishing cycle sheaf $\varphi_{\text{Tr}\,(W)}\mathbb{Q}$ of the regular function 
$$\text{Tr}\,(W):\X(d)\to\mathbb{A}^1_{\mathbb{C}}$$
on the stack $\X(d)$ of representation of $Q$ of a given dimension $d$.

For any Calabi-Yau $3$-fold $Y,$ the stack $\mathfrak{M}_\beta$ is locally analytically described by $\text{crit}\,(\text{Tr}\,(W))\subset \X(d)$ for a quiver with potential $(Q,W)$ \cite{j}, \cite{t} and the DT invariants are defined using the sheaves $\varphi_{\text{Tr}\,(W)}\mathbb{Q}$.
This description is global for $Y=\mathbb{A}^3_{\mathbb{C}}$ and $\beta=d\in\mathbb{N}$. Consider the quiver $Q_3$ with one vertex and three loops $x,y,$ and $z$ and potential $W_3=xyz-xzy$. Then \[\mathfrak{M}_\beta\cong\text{crit}\,(\text{Tr}\,(W_3))\subset \X(d).\]
It is thus worthwhile to study the DT theory of quivers with potential and try to generalize the constructions and results to the general case of a Calabi-Yau $3$-fold.

\subsection{Cohomological Hall algebras.} \label{cohanot}

Let $I$ be the set of vertices of $Q$,
let $\theta\in \mathbb{Q}^{I}$ be a generic King stability condition, let $\mu\in \mathbb{Q}$ a slope, let $\Lambda_\mu\subset \mathbb{N}^{I}$ be the subset of dimension vectors of slope $\mu$, and let $d\in \Lambda_\mu$. 
Denote
by \[\X(d)^{\text{ss}}\subset \X(d):=R(d)/G(d)\] the stack of $\theta$-semistable representations of slope $\mu$. 
The Cohomological Hall algebra (CoHA), constructed by Kontsevich--Soibelman \cite{ks}, is an algebra with underlying $\Lambda_\mu$-graded
vector space:
$$\text{CoHA}\,(Q, W)_{\mu}:=\bigoplus_{d\in \Lambda_\mu}H^{\cdot}\left(\X(d)^{\text{ss}}, \varphi_{\text{Tr}\,W}\mathbb{Q}\right),$$ where the multiplication $m=p_*q^*$ is defined using the maps \begin{equation}\label{pq}
    \X(d)^{\text{ss}}\times \X(e)^{\text{ss}}\xleftarrow{q} \X(d,e)^{\text{ss}}\xrightarrow{p} \X(d+e)^{\text{ss}}
    \end{equation} from the stack $\X(d,e)$ parametrizing pairs of representations $A\subset B$ with $A$ of dimension $d$ and $B$ of dimension $d+e$. 
Consider the regular function \[\text{Tr}(W):\X(d)\to\mathbb{A}^1_{\mathbb{C}}.\] Assume zero is its only critical value. The critical locus of $\text{Tr}(W)$ is $\X(Q,W,d)$, the moduli of representations of dimension $d$ of the Jacobi algebra $$\text{Jac}(Q,W):=\mathbb{C} Q\Big\slash\left(\frac{\partial W}{\partial e}, e\in E\right),$$ so the vector space $H^{\cdot}(\X(d),\varphi_{\text{Tr}\,W}\mathbb{Q})$ is the critical cohomology of the (usually singular) space $\X(Q,W,d)$ with coefficients in (a shift of) a perverse sheaf.
Using framed quivers, Davison--Meinhardt \cite{dm} and Soibelman \cite{s} constructed representations of these algebras.

For any quiver $Q$, there is a tripled quiver with potential $\left(\widetilde{Q}, \widetilde{W}\right)$ whose CoHA recovers the preprojective Hall algebra of $Q$ as defined by Schiffmann--Vasserot \cite{sv2}, Yang--Zhao \cite{yz}, see \cite{rs}.
In \cite{d}, Davison conjectured that an $\mathbb{C}^*$-equivariant version of $\text{CoHA}\left(\widetilde{Q}, \widetilde{W}\right)$ is the positive half of the Maulik--Okounkov Yangian $Y_{\text{MO}}$.


\subsection{Categorical and K-theoretic Hall algebras.} 
In \cite[Section 8.1]{ks}, Kontsevich--Soibelman propose the category of singularities 
\[D_{\text{sg}}(\X(d)_0):=D^b\text{Coh}(\X(d)_0)\slash \text{Perf}(\X(d)_0)\]
as a categorification of the critical cohomology $H^{\cdot}(\X(d), \varphi_{\text{Tr}\,W}\mathbb{Q})$, see Efimov's work \cite{e}. 
The category of singularities is equivalent to the category of matrix factorizations $\text{MF}(\X(d), W)$ for the regular function $\text{Tr}\,W$. 

Consider the torus $\left(\mathbb{C}^*\right)^E$ whose factor corresponding to $e\in E$ acts on $R(d)$ by scaling the linear map corresponding to $e$. Let $T\subset (\mathbb{C}^*)^{E}$ be a torus under which $W$ is invariant. 
We use the notations introduced in Subsection \ref{cohanot}.

\begin{thm}\label{thm:1}
Consider the $\Lambda_\mu$-graded category
$$\text{HA}_T(Q,W)_{\mu}:=\bigoplus_{d\in\Lambda_\mu} D_{\text{sg}, T}\left(\X(d)^{\text{ss}}_0\right).$$
Then $\text{HA}_T(Q,W)_{\mu}$ is monoidal with multiplication $m:=p_*q^*$, where $p$ and $q$ are the maps in \eqref{pq}. The underlying category is called the categorical Hall algebra (HA) of $(Q,W)$. The Grothendieck group is called the K-theoretic Hall algebra (KHA) of $(Q,W)$. 
\end{thm}

In analogy to cohomology, we may call $K_0\left(D_{\text{sg}}\left(\X(d)_0\right)\right)$ the critical K-theory of $\X(Q,W,d)$ and may denote it by $K_{\text{crit}}\left(\X(Q,W,d)\right)$. We denote the full Hall algebra (for the zero stability condition) by $\text{HA}_T(Q,W)$ and $\text{KHA}_T(Q,W)$.

Assume that there is a $\mathbb{C}^*\subset \left(\mathbb{C}^*\right)^E$ such that $\text{Tr}\,W_d$ has weight $2$ with respect to $\mathbb{C}^*$ for any $d\in\mathbb{N}^I$. Then one can consider graded matrix factorization categories $\text{MF}^{\text{gr}}(\X(d), W)$ which are equivalent to categories of graded singularities $D_{\text{sg}}^{\text{gr}}\left(\X(d)^{\text{ss}}_0\right)$. One can define a graded version of HA with underlying
$\Lambda_\mu$-graded category
$$\text{HA}^{\text{gr}}_T(Q,W)_{\mu}:=\bigoplus_{d\in\Lambda_\mu} D^{\text{gr}}_{\text{sg}, T}\left(\X(d)^{\text{ss}}_0\right)$$
and multiplication $m:=p_*q^*$. Its Grothendieck group is called $\text{KHA}^{\text{gr}}$.

\subsection{Preprojective Hall algebras}
Let $Q$ be a quiver and consider the tripled quiver $\left(\widetilde{Q}, \widetilde{W}\right)$. 
Consider the preprojective Hall algebra defined by Varagnolo--Vasserot:
\[\text{HA}_T(Q):=\bigoplus_{d\in\mathbb{N}^I} D^b_T\left(\mathfrak{P}(d)\right),\]
where $\mathfrak{P}(d)$ is the stack of representations of dimension $d$ of the preprojective algebra of $Q$. 
Using Isik's equivalence \cite{i}, there is an equivalence of underlying categories
\[\text{HA}^{\text{gr}}_T\left(\widetilde{Q}, \widetilde{W}\right)\cong \text{HA}_T(Q),\]
where $\text{HA}^{\text{gr}}$ is defined using
 a natural $\mathbb{C}^*\subset \left(\mathbb{C}^*\right)^{\widetilde{E}}$ such that $\text{Tr}\,\widetilde{W}_d$ is homogeneous of weight $2$ for any $d\in\mathbb{N}^I$.
The multiplications differ by conjugation by an equivariant parameter, see \cite{VV} and Subsection \ref{comparisontwoKHAs}. Using \cite[Corollary 3.13]{T2}, the categories 
$\text{MF}_T(\X(d), W)$ and $\text{MF}^{\text{gr}}_T(\X(d), W)$ have the same Grothendieck group, so 
there is an isomorphism \begin{equation}\label{isogra}
    \text{KHA}^{\text{gr}}_T\left(\widetilde{Q}, \widetilde{W}\right)\cong \text{KHA}_T\left(\widetilde{Q}, \widetilde{W}\right).\end{equation}

\subsection{Representations of the KHA}\label{rKHA} There are representations of $\text{KHA}_T(Q,W)_{\mu}$ on critical $K$-theory spaces associated to moduli of framed representations
$$\bigoplus_{d\in\Lambda_\mu} K_0^T\left(D_{\text{sg}}\left(\X(f,d)^{\text{ss}}_0\right)\right).$$ There are analogous representations of $\text{KHA}^{\text{gr}}_T(Q,W)_{\mu}$.
These representations are analogous to the ones constructed in \cite{dm}, \cite{s} in cohomology.

Further, for a quiver $Q$, there are representations of $\text{KHA}_T^{\text{gr}}\left(\widetilde{Q},\widetilde{W}\right)\cong \text{KHA}_T(Q)$ on the $K$-theory of Nakajima quiver varieties
$$\bigoplus_{d\in\mathbb{N}^I} K_0^T\left(N(f,d)\right).$$

Quantum affine algebras also naturally act on Nakajima quiver varieties \cite{n}. Analogous to Davison's conjecture \cite{d}, 
we expect preprojective KHAs to be related to positive parts 
of quantum affine algebras.

\begin{conjecture}\label{conjecture}
\textit{Consider the torus $\mathbb{C}^*$ scaling the linear maps corresponding to edges of the doubled quiver $Q^d$ with weight $1$ and scaling the linear maps corresponding to
edges of the loops $\omega_i$ with weight $-2$, see Subsection \ref{tripledef} for the definitions of the doubled quiver $Q^d$ and of the
tripled quiver $\widetilde{Q}$. After possibly tensoring with $\mathbb{C}(q)\cong \left(\text{Frac}\,K_0(B\mathbb{C}^*)\right)\otimes\mathbb{C}$,
there is an isomorphism $$\text{KHA}_{\C^*}\left(Q\right)\cong U^{>}_q\left(\widehat{\mathfrak{g}_Q}\right),$$  where the right hand side is the positive part of Okounkov--Smirnov affine quantum algebra.}
\end{conjecture}

The conjecture is true for finite and affine type quivers except $A_1^{(1)}$, see \cite{VV}. In these cases, $\mathfrak{g}_Q$ is the Kac-Moody algebra of $Q$, but for general quivers $Q$, the Lie algebra $\mathfrak{g}_Q$ is strictly larger than the Kac-Moody algebra of $Q$. 

\subsection{The Jordan quiver}
For the Jordan quiver $Q$
\begin{tikzcd}
1 \arrow[out=0,in=90,loop,swap,"f"]
\end{tikzcd}
with one vertex and one loop, the tripled quiver $\widetilde{Q}$
\begin{tikzcd}
1 \arrow[out=0,in=90,loop,swap,"x"]
  \arrow[out=120,in=210,loop,swap,"y"]
  \arrow[out=240,in=330,loop,swap,"z"]
\end{tikzcd}
has one vertex, three loops, and potential $\widetilde{W}=xyz-xzy.$ Let $T\subset \left(\mathbb{C}^*\right)^3$ be a torus which fixes $W$. Consider the action of $\mathbb{C}^*$ which scales the linear map corresponding to $z$ with weight $2$. We consider graded matrix factorization with respect to $\mathbb{C}^*$. 
By Isik's theorem \cite{i}, there is an equivalence
$$\text{HA}_T^{\text{gr}}\left(\widetilde{Q},\widetilde{W}\right)\cong
\bigoplus_{d\geq 0} D^b_T\left(\mathfrak{C}(d)\right),$$ where $\mathfrak{C}(d)$ is the stack of commuting matrices of dimension $d$ and the right hand side has an algebra structure defined by correspondences \cite{VV}, \cite{sv1}. 
The framed representations of the $\text{KHA}_T\left(\widetilde{Q},\widetilde{W}\right)$ from Subsection \ref{rKHA}
for the vector $f=1$ are 
$$\bigoplus_{d\geq 0} K_{\text{crit}}^T\left(\text{Hilb}\left(\mathbb{A}^3_{\C},d\right)\right).$$ Further, one constructs representations of $\text{KHA}_T\left(\widetilde{Q},\widetilde{W}\right)$ on the vector space
\begin{equation}\label{hilbert}
\bigoplus_{d\geq 0}K_0^T\left(\text{Hilb}\left(\mathbb{A}^2_{\C},d\right)\right).
\end{equation}
Schiffmann--Vasserot \cite{sv1} and Feigin--Tsymbaliuk \cite{ft} construct representations of $U_{q, t}\left(\widehat{\widehat{\mathfrak{gl}_1}}\right)$, 
the Drinfeld double of a subalgebra of $\text{KHA}_T\left(\widetilde{Q},\widetilde{W}\right)$, on the vector space \eqref{hilbert}.

\subsection{Wall-crossing}
The CoHA satisfies a wall-crossing theorem by work of Davison--Meinhardt \cite{dm}. We prove an analogous result for KHAs of quivers with potential which satisfy a Künneth-type assumption, see Subsection \ref{kunn}:
\[\text{KHA}(Q,W)\xrightarrow{\sim} \bigotimes_{\mu\in\mathbb{Q}} \text{KHA}(Q,W)_{\mu}.\]
It is an advantage that we can formulate an analogous categorical statement because by general principles it suffices to check the categorical statement for the zero potential. In this case, the statement follows from work of Halpern-Leistner \cite{hl}, Ballard--Favero--Katzarkov \cite{bfk} on
semiorthogonal decompositions in GIT.

\subsection{Further properties of the KHA} There is a PBW theorem for CoHAs \cite{dm} for all symmetric quivers $Q$ with potential. For KHAs, we expect such a theorem for $\left(\widetilde{Q}, \widetilde{W}\right)$ by Conjecture \ref{conjecture}. 
In \cite[Section 7]{P}, we prove a PBW theorem for KHAs for all pairs $(Q,W)$ with $Q$ symmetric. 

In \cite{P2}, inspired by explicit computations of KHAs due to Negu\c{t} \cite{ne1}, \cite{ne3}, we construct a Drinfeld double Hopf algebra of KHAs for a class of quivers with potential $(Q,W)$ satisfying a Künneth-type assumption. This class includes all tripled quivers $\left(\widetilde{Q}, \widetilde{W}\right)$.

\subsection{Acknowledgements} I would like to thank my PhD advisor Davesh Maulik for suggesting the problems discussed in the present paper and for his constant help and encouragement throughout the project. I would like to thank Ben Davison, Pavel Etingof, Daniel Halpern-Leistner, Andrei Negu\c{t}, Andrei Okounkov, Yukinobu Toda, and Eric Vasserot for useful conversations about the project.
I thank the referees for many useful comments.

\subsection{Outline of the paper} In Section \ref{2}, we review notions about quivers with potential, semiorthogonal decompositions, categories of singularities, and matrix factorizations. In Section \ref{3}, we prove Theorem \ref{thm:1} and discuss some examples of KHAs. In Section \ref{7}, we construct the representations of the KHA mentioned in Subsection \ref{rKHA}. In Section \ref{4}, we prove wall-crossing theorems for categorical and $K$-theoretic HAs.

\subsection{Notations and conventions} All the schemes and stacks considered are over $\mathbb{C}$. Let $X$ be a scheme or stack.
We denote by $D^b(X)$ the derived category of coherent sheaves, by $\text{Perf}(X)\subset D^b(X)$ its subcategory of perfect complexes, and by $D_{\text{sg}}(X)$ the category of singularities. All functors considered, such as pullback and pushforward, are derived.
We denote by $K_{i}(X)$ the $K$-theory of the category $\text{Perf}(X)$ and by $G_{i}(X)$ the $K$-theory of the category $D^b(X)$. For a regular immersion $\iota:\X\hookrightarrow \X'$, denote by $N_\iota$ the normal bundle of $\X$ in $\X'$. For the purposes of this paper, a smooth quotient stack will have the form
\begin{equation}\label{smqst}
    \mathcal{X}\cong [A/G],
\end{equation} where $A$ is a (quasi-)affine smooth variety and $G$ is a reductive group.

The categories considered are dg and we denote by $\otimes$ the product of dg categories \cite[Subsections 2.2 and 2.3]{K}.

We assume that the quivers with potential $(Q,W)$ considered are such that the regular functions $\text{Tr}\,W: \X(d)\to \mathbb{A}^1_{\mathbb{C}}$ have zero as the only critical value. 
We denote by $\text{MF}(\X(d), W)$ the category of matrix factorizations for the regular function $\text{Tr}\,W$. The zero fiber $\X(d)_0$ of $\text{Tr}\,W$ is derived.

\section{Background material}\label{2}

\subsection{Quivers with potential}\label{quivers}
Let $Q=(I, E, s, t)$ be a quiver with vertex set $I$, edge set $E$, and source and target maps $s, t: E\to I$. Let $d=\left(d^i\right)_{i\in I}\in \mathbb{N}^{I}$ be a dimension vector of $Q$. Consider vector spaces $V^i$ of dimension $d^i$.
Consider the reductive group $G(d)$ and its representation $R(d)$:
\begin{align*}
    G(d)&:=\prod_{I\in I} GL\left(V^i\right),\\
    R(d)&:=\prod_{e\in E} \text{Hom}\,\left(V^{s(e)}, V^{t(e)}\right).
\end{align*}
Define the quotient stack of representation of $Q$ of dimension $d$: \[\X(d):=R(d)\slash G(d).\]
A potential $W$ is a linear combination of cycles in $Q$. A potential determines a regular function: 
$$\text{Tr}\,(W):\X(d)\to\mathbb{A}^1_{\mathbb{C}}.$$ We will assume throughout the paper that $0$ is the only critical value. The critical locus of this function is the moduli of representations of the Jacobi algebra $\text{Jac}\,(Q,W):=\mathbb{C}Q\slash \mathcal{J}$, where $\mathbb{C}Q$ is the path algebra of $Q$ and $\mathcal{J}$ is the two-sided ideal in $\mathbb{C}Q$ generated by the derivatives $\frac{\partial W}{\partial e}$ of $W$ along all edges $e\in E$.

\subsection{King stability conditions}
Given a tuple $\theta=\left(\theta^i\right)_{i\in I}\in\mathbb{Q}^{I}$, we define the slope function on a dimension vector $d\in\mathbb{N}^{I}\setminus\{0\}$ by $$\tau(d):=\frac{\sum_{i\in I}\theta^id^i}{\sum_{i\in I}d^i}\in\mathbb{Q}.$$ 
For a slope $\mu\in \mathbb{Q}$, let $\Lambda_{\mu}\subset \mathbb{N}^{I}$ be the monoid of dimension vectors $d$ with $\tau(d)=\mu$ together with $d=0$.
Call a representation $V$ of $Q$ \textit{($\theta$-)(semi)stable} if for every proper subrepresentation $W\subset V$, we have that \[\tau(W)<(\leq) \tau(V).\] The locus of stable representations $R(d)^{\text{s}}$ and semistable representations $R(d)^{\text{ss}}$ inside $R(d)$ are open. We consider the moduli stack \[\X(d)^{\text{ss}}:=R(d)^{\text{ss}}/G(d)\] of semistable representations of dimension $d$.

\subsection{Moduli of framed representations}\label{framed}
Fix a framing vector $f\in \mathbb{N}^{I}$. We define a new quiver $Q^f=(I^f, E^f)$ with $I^f=I\sqcup \{\infty\}$ and $E^f$ contains $E$ and $f^i$ edges from $\infty$ to the vertex $i\in I$.
The dimension vector $d\in\mathbb{N}^{I}$ can be extended to a dimension vector for the new quiver \[\widetilde{d}:=(1,d)\in\mathbb{N}^{I^f}=\mathbb{N}\times \mathbb{N}^{I}.\]
Fix a slope $\mu\in\mathbb{Q}$. Define $\theta'=\mu+\varepsilon$ for a small positive rational number $\varepsilon>0$.
The stability condition $\theta$ is extended to a stability condition for the quiver $Q^f$:
\[\theta^f:=(\theta', \theta)\in \mathbb{Q}^{I^f}.\] 

\subsection{The tripled quiver}\label{tripledef}
The following construction was introduced by Ginzburg \cite{gi} and it is used in conjunction with dimensional reduction to obtain representations of a preprojective CoHA or KHA on the cohomology or $K$-theory of Nakajima quiver varieties.

Let $Q=(I,E)$ be a quiver. For an edge $e$, let $\overline{e}$ be the edge of opposite orientation. Let $\overline{E}:=\{\overline{e}|\,e\in E\}$.
The double quiver $Q^d=(I, E^d)$ has edge set $E^d:=E\cup \overline{E}$. For every $i\in I$, denote by $\omega_i$ a loop at $i$. The tripled quiver $\widetilde{Q}=\left(I, \widetilde{E}\right)$ has vertex set $I$ and $\widetilde{E}=E^d\sqcup \{\omega_i|\,i\in I\}.$ The potential $\widetilde{W}$ is defined by $$\widetilde{W}:=\sum_{e\in E} \omega_{s(e)}[\bar{e}, e].$$

\subsection{Nakajima quiver varieties}\label{nakquivvar}

Let $Q$ be a quiver, $d\in\mathbb{N}^I$ a dimension vector, $\theta\in\mathbb{Q}^I$ a stability condition, $\mu\in\mathbb{Q}$,
and $f\in\mathbb{N}^I$ a framing vector. Extend $\theta$ to the stability condition $\theta^f$ for $Q^f$ as in Subsection \ref{framed}. Associated to $\theta$, there is a character \[\chi_{\theta}:=\prod_{i\in I} \det(g^i)^{m\theta^i}: G(d)\to\mathbb{C}^*\] for $m$ a positive integer such that $m\theta^i$ are all integers.
The action of $G(d)\cong G(1,d)/\mathbb{C}^*$ on $R(1,d)$ induces a moment map:
$$\mu: T^*R(1,d)\to \mathfrak{g}(d)^{\vee}\cong \mathfrak{g}(d).$$
Define the Nakajima quiver variety $N(f,d)$ by the GIT quotient:
$$N(f,d):=\mu^{-1}(0)\sslash_{\chi_{\theta}}G(d).$$
There is also a description of Nakajima quiver varieties using the framed quiver in Subsection \ref{framed} given by Crawley-Boevey \cite[Section 1]{CB}. 

\subsection{Semiorthogonal decompositions.}\label{sodadjoint}
Let $\A$ be a triangulated category and let $\A_i\subset \A$ be full triangulated subcategories for $1\leq i\leq n$. We say that $\A$ has a semiorthogonal decomposition: $$\A=\langle \A_m,\cdots,\A_1\rangle$$ if
for every objects $A_i\in\mathcal{A}_i$ and $A_j\in\mathcal{A}_j$ and $i<j$ we have
$\text{RHom}\,(A_i,A_j)=0$, 
and the smallest full triangulated subcategory of $\mathcal{A}$ containing $\mathcal{A}_i$ for $1\leq i\leq m$ is $\mathcal{A}$.

Let $\mathcal{B}$ be a triangulated subcategory of $\mathcal{A}$. There exists a semiorthogonal decomposition $\A=\langle \mathcal{C}, \mathcal{B}\rangle$ if and only if the inclusion $\mathcal{B}\hookrightarrow \A$ has a right adjoint $\A\to \mathcal{B}$. If this happens, we say that $\mathcal{B}$ is \textit{right admissible} in $\mathcal{A}$.

\subsection{Window categories.}\label{window}
\subsubsection{} Let $\X=X/G$ be a quotient stack where $G$ be a reductive group and $X$ is a smooth affine variety with a $G$ action.

For pairs $(\lambda, Z)$ with $\lambda$ a cocharacter of $G$ and $Z$ a connected component of $X^\lambda$, consider the diagram \begin{equation}\label{pqq}
\mathcal{Z}:=Z/L\xleftarrow{q}\mathcal{S}:=S/P
\xrightarrow{p} \X,
\end{equation}
where $S\subset X$ is the subset of points $x$ such that $\lim_{z\to 0} \lambda(z)x\in Z$, and $L$ and $P$ are the Levi and parabolic groups corresponding to $\lambda$. The map $q$ is an affine bundle map and the map $p$ is proper.  We say that $(\lambda, Z)$ is a \textit{Kempf-Ness stratum} if the map $p$ is a closed immersion. The map $p$ is always a closed immersion if $G$ is abelian.

\subsubsection{}\label{KNstrata} Consider locally closed substacks $\mathcal{S}_i\subset \X$ indexed by $i\in I$ for $I$ a partially ordered set such that \[\mathcal{S}_i\subset \X\setminus \bigcup_{j<i}\mathcal{S}_j\] is a Kempf-Ness stratum. We denote by 
\begin{align*}
    \mathcal{U}&:=\bigcup_{i\in I}\mathcal{S}_i,\\
    \X^{\text{ss}}&:=\X\setminus\mathcal{U}.
\end{align*}
It might happen that $\X^{\text{ss}}$ is empty.
An example of such a stratification is given by the usual Kempf-Ness strata $\mathcal{S}_i$ for $i\in I$ and the semistable stack $\X^{\text{ss}}\subset \X$ with respect to a linearization $\mathcal{L}$ on $\X$. 

\subsubsection{} We continue with the notation from the previous Subsection. Halpern-Leistner \cite{hl} constructed categories $\mathbb{G}_w\subset D^b(\X)$ which are equivalent to $D^b\left(\X^{\text{ss}}\right)$ under the restriction map $j:\X^{\text{ss}}\hookrightarrow \X$, which we now explain.



Let $i\in I$. Assume that $\mathcal{S}_i$ is an attracting locus for $(\lambda_i, Z_i)$.
Consider the inclusion map $j_i:Z_i\hookrightarrow X$. 
For $w\in \mathbb{Z}$, let $D^b(\mathcal{Z})_w$ be the subcategory of $D^b(\mathcal{Z})$ of complexes on which $\lambda_i$ acts with weight $w$.
 Define $n_i=\langle \lambda_i^{-1}, j_i^*\left(\det N_{p_i}\right)\rangle$, where $N_{p_i}$ is the normal bundle of $\mathcal{S}_i$ in $\X$.
 Choose $w_i\in \mathbb{Z}$ and define 
 $$\mathbb{G}_w:=\{F\in D^b(\X)\text{ such that } w_i\leq \langle \lambda_i, j_i^*F\rangle\leq w_i+n_i-1\}.$$
 In \cite[Theorem 2.10, Amplification 2.11]{hl}, Halpern-Leistner constructs a semiorthogonal decomposition: 
\begin{equation}\label{hl}
    D^b(\X)=\big\langle p_{i*}q_i^* D^b(\mathcal{Z}_i)_{v_i},\mathbb{G}_w,  p_{i*}q_i^* D^b(\mathcal{Z}_i)_{t_i}\big\rangle,
\end{equation}
 where the categories on the left hand side of $\mathbb{G}_w$ are after all $i\in I$ and all $v_i<w_i$ and the categories on the right hand side of $\mathbb{G}_w$ are after all $i\in I$ and all $t_i\geq w_i$. The functors $q_i^*$ and $p_{i*}$ are fully faithful on $D^b(\mathcal{Z}_i)_{v_i}$ and $q_i^*D^b(\mathcal{Z}_i)_{v_i}$, respectively, for $i\in I$ and $v_i$ as above.
The restriction functor 
$j^*:D^b(\X)\to D^b\left(\X^{\text{ss}}\right)$ induces an equivalence of categories: $$j^*:\mathbb{G}_w\xrightarrow{\sim} D^b\left(\X^{\text{ss}}\right).$$

\subsection{Categories of singularities and matrix factorizations.}\label{singul}
A reference for this section is \cite[Section 2.2]{T3}.
Let $Y$ be an affine scheme with an action of a reductive group $G$. Consider the quotient stack
$\mathcal{Y}=Y/G$.
The category of singularities of $\mathcal{Y}$ is a triangulated category defined as the quotient of triangulated categories
$$D_{\text{sg}}(\mathcal{Y}):=D^b(\mathcal{Y})/\text{Perf}(\mathcal{Y}),$$
where $\text{Perf}(\mathcal{Y})\subset D^b(\mathcal{Y})$ is the full subcategory of perfect complexes. 
If $\mathcal{Y}$ is smooth, the category of singularities is trivial. We have an exact sequence
\[K_0(\mathcal{Y})\to G_0(\mathcal{Y})\to K_0\left(D_{\text{sg}}(\mathcal{Y})\right)\to 0.\]

Let $\X=X/G$ be a smooth quotient stack with $X$ an affine scheme and consider
a regular function 
\[f:\mathcal{X}\to\mathbb{A}^1_{\mathbb{C}}.\]
Consider the category of matrix factorizations $\text{MF}(\X, f)$. It has objects $(\mathbb{Z}/2\mathbb{Z})\times G$-equivariant factorizations $(P, d_P)$, where $P$ is a $G$-equivariant coherent sheaf, $\langle 1\rangle$ is the twist corresponding to a non-trivial $\mathbb{Z}/2\mathbb{Z}$-character on $X$, and \[d_P: P\to P\langle 1\rangle\] with $d_P\circ d_P=f$. Alternatively, the objects of $\text{MF}(\X, f)$ are tuplets
\[(F, G, \alpha: F\to G, \beta: G\to F),\]
where $F$ and $G$ are $ G$-equivariant coherent sheaves, $\alpha$ and $\beta$ are $G$-equivariant morphisms
with $\alpha\circ\beta$ and $\beta\circ\alpha$ are multiplication by $f$. 
By a theorem of Orlov \cite{o2},  there is an equivalence 
\[D_{\text{sg}}(\X_0)\cong \text{MF}(\X, f).\]
Recall that for $f=0$, the fiber $\X_0$ is derived.
We will freely switch between $D_{\text{sg}}$ and $\text{MF}$ throughout this paper.

For a triangulated subcategory $\mathcal{A}$ of $D^b(\X)$, define $\text{MF}(\mathcal{A}, f)$ as the full subcategory of $\text{MF}(\mathcal{X}, f)$ with objects pairs $(P, d_P)$ with $P$ in $\mathcal{A}$. We explain next that semiorthogonal decompositions for the ambient smooth stack induce semiorthogonal decompositions for matrix factorizations, see also \cite[Lemma 1.18]{hlp}.

\begin{prop}\label{prop1}
Let $I$ be a totally ordered set and consider a semiorthogonal decomposition
\[D^b(\mathcal{X})=\big\langle \mathcal{A}_i\big\rangle_{i\in I}.\]
There is a semiorthogonal decomposition
\[\text{MF}(\mathcal{X}, f)=\big\langle \text{MF}(\mathcal{A}_i, f)\big\rangle_{i\in I}.\]
\end{prop}

\begin{proof}
Assume for simplicity that the semiorthogonal decomposition is \[D^b(\mathcal{X})=\big\langle \mathcal{A}_1, \mathcal{A}_2\big\rangle.\]
Consider an object \[E=\left(\alpha: F\rightleftarrows G: \beta\right)\] in $\text{MF}(\mathcal{X}, f)$ and consider $F_i, G_i\in \mathcal{A}_i$ such that
\begin{align*}
    F_2\to F&\to F_1\xrightarrow{[1]}\\
    G_2\to G&\to G_1\xrightarrow{[1]}.
\end{align*}
The map $\alpha: F\to G$ induces a map $\alpha: F_2\to G$ and thus a map 
\[\alpha_2: F_2\to G_2\] because $\text{RHom}(F_2, G_1)=0$. Further, it induces a map $\alpha: F\to G_1$ and thus a map
\[\alpha_1: F_1\to G_1\] because $\text{RHom}(F_2, G_1)=0$. Similarly, there are induced maps $\beta_i: G_i\to F_i$ for $i=1,2$. The tuplets
\[E_i:=\left(\alpha_i: F_i\rightleftarrows G_i: \beta_i\right)\] are in $\text{MF}(\mathcal{A}_i, f)$. 
There is a distinguished triangle 
\[E_1\to E\to E_2\xrightarrow{[1]}.\]
The orthogonality claim is immediate.
\end{proof}

We say that $f$ satisfies \textbf{Assumption A} if there is an extra action of $\mathbb{C}^*$ on $X$ which commutes with the action of $G$ such that $f$ is $\mathbb{C}^*$-equivariant of weight $2$. Denote by $(1)$ the twist by the character \[\text{pr}_2:G\times\mathbb{C}^*\to\mathbb{C}^*.\]
Consider the category of graded matrix factorizations $\text{MF}^{\text{gr}}(\mathcal{X}, f)$. It has objects pairs $(P, d_P)$ with $P$ an equivariant $G\times\mathbb{C}^*$-sheaf on $X$ and $d_P:P\to P(1)$ a $G\times\mathbb{C}^*$-equivariant morphism. For $f$ zero and the trivial $\mathbb{C}^*$-action on $\X$, we have that
\[\text{MF}^{\text{gr}}(\X,0)\cong D^b(\X),\]
see \cite[Remark 2.3.7]{T3}.
For a triangulated subcategory $\mathcal{B}$ of $D^b_{\mathbb{C}^*}(\X)$, define $\text{MF}^{\text{gr}}(\mathcal{B}, f)$ as the full subcategory of $\text{MF}^{\text{gr}}(\mathcal{X}, f)$ with objects pairs $(P, d_P)$ with $P$ in $\mathcal{B}$. The same argument used in Proposition \ref{prop1} shows that:

\begin{prop}\label{prop2}
Let $I$ be a totally ordered set and consider a semiorthogonal decomposition
\[D^b_{\mathbb{C}^*}(\mathcal{X})=\big\langle \mathcal{B}_i\big\rangle_{i\in I}.\]
There is a semiorthogonal decomposition
\[\text{MF}^{\text{gr}}(\mathcal{X}, f)=\big\langle \text{MF}^{\text{gr}}(\mathcal{B}_i, f)\big\rangle_{i\in I}.\]
\end{prop}

\subsection{Functoriality of categories of singularities.}\label{functo}
References for this subsection are \cite{pv}, \cite{T3}. Let $\mathcal{X}$, $\mathcal{X}'$ be smooth quotient stacks, see \eqref{smqst}, with a map $\alpha: \X'\to \X$. Let \[f:\mathcal{X}\to\mathbb{A}^1_{\mathbb{C}}\] be a regular function and consider $f':=f\alpha:\X'\to \mathbb{A}^1_{\mathbb{C}}.$ Assume that $0$ is the only critical value for each of $f$ and $f'$. There is a functor
\begin{align*}
    \alpha^*: \text{MF}(\X, f)&\to \text{MF}\left(\X', f'\right),\\
    (P, d_P)&\mapsto\left(\alpha^*P, \alpha^*d_P\right).
    \end{align*}
    If $\alpha$ is proper, there is a functor
\begin{align*}
\alpha_*: \text{MF}\left(\X', f'\right)&\to \text{MF}(\X, f),\\
\left(P', d_{P'}\right)&\mapsto\left(\alpha_*P', \alpha_*d_{P'}\right).
\end{align*}
Assume there are $\mathbb{C}^*$-actions on $\X'$ and $\X$ such that $\alpha$ is $\mathbb{C}^*$-equivariant and such that $\X$ and $\X'$ satisfy Assumption A with respect to the $\mathbb{C}^*$-action.
There are pullback and pushforward functors
\begin{align*}
    \alpha^*&: \text{MF}^{\text{gr}}(\X, f)\to \text{MF}^{\text{gr}}\left(\X', f'\right),\\
    \alpha_*&: \text{MF}^{\text{gr}}\left(\X', f'\right)\to \text{MF}^{\text{gr}}(\X, f).
    \end{align*}
There are also such functors for categories of singularities. These pullback and pushforward functors satisfy properties as those for derived categories, for example proper base change for cartesian diagrams.

\subsection{Thom-Sebastiani theorem}\label{TSiso}
Let $\X$ and $\mathcal{Y}$ be smooth quotient stacks with regular functions 
\begin{align*}
    f&:\X\to\mathbb{A}^1_{\mathbb{C}},\\
    g&:\mathcal{Y}\to\mathbb{A}^1_{\mathbb{C}},\\
    f+g&:\X\times\mathcal{Y}\to\mathbb{A}^1_{\mathbb{C}}.
\end{align*}
Consider the functors induced by exterior tensor product, see \cite[Definition 3.22]{BFK2}:
\begin{align*}
    \text{TS}&:\text{MF}^{\text{gr}}(\X, f)\otimes \text{MF}^{\text{gr}}(\mathcal{Y}, g)\to \text{MF}^{\text{gr}}(\X\times\mathcal{Y}, f+g),\\
    \text{TS}&:\text{MF}(\X, f)\otimes \text{MF}(\mathcal{Y}, g)\to \text{MF}(\X\times\mathcal{Y}, f+g).
\end{align*}
To define the first functor above, we assume that $\X$ and $\mathcal{Y}$ satisfy Assumption A and denote the corresponding tori by $\mathbb{C}^*_1$ and $\mathbb{C}^*_2$; the corresponding $\mathbb{C}^*$ on $\X\times\mathcal{Y}$ is the diagonal $\mathbb{C}^*\hookrightarrow \mathbb{C}^*_1\times \mathbb{C}^*_2$.

The Thom-Sebastiani theorem says that the first functor is an equivalence \cite[Section 3 and Section 5.1]{BFK2}, \cite[Section 2.5]{EP}:
\begin{equation}\label{thomseba}\text{MF}^{\text{gr}}(\X, f)\otimes \text{MF}^{\text{gr}}(\mathcal{Y}, g)\cong \text{MF}^{\text{gr}}(\X\times\mathcal{Y}, f+g).\end{equation}
There is also a version when a version of the second one is an equivalence \cite{Pr}. 
When using categories of singularities, the Thom-Sebastiani functor is induced by pushforward along $i:\X_0\times\mathcal{Y}_0\hookrightarrow(\X\times\mathcal{Y})_0$, see \cite[Theorem 4.13]{Pr}:
\[i_*: D_{\text{sg}}(\X_0)\otimes D_{\text{sg}}(\mathcal{Y}_0)\to D_{\text{sg}}\left((\X\times\mathcal{Y})_0\right).\]
There are maps
\begin{align*}
    \text{TS}&:K_0\left(\text{MF}^{\text{gr}}(\X, f)\right)\otimes K_0\left(\text{MF}^{\text{gr}}(\mathcal{Y}, g)\right)\to K_0\left(\text{MF}^{\text{gr}}(\X\times\mathcal{Y}, f+g)\right),\\
    \text{TS}&:K_0\left(\text{MF}(\X, f)\right)\otimes K_0\left(\text{MF}(\mathcal{Y}, g)\right)\to K_0\left(\text{MF}(\X\times\mathcal{Y}, f+g)\right).
\end{align*}
In general, these maps are not isomorphisms.

\subsection{Dimensional reduction}
\label{dimred0}
Let $X$ be a smooth affine scheme with an action of a reductive group $G$, let $\X=X/G$, and let $E$ be a $G$-equivariant vector bundle on $X$. Let $\mathbb{C}^*$ act on the fibers of $E$ with weight $2$ and consider $s\in \Gamma(X, E)$ a section of $E$ of $\mathbb{C}^*$-weight $2$.
It induces a map $\partial: E^{\vee}\to \mathcal{O}_X$. Consider the Köszul stack
\[\mathfrak{P}:=\text{Spec}\left(\mathcal{O}_X\left[E^{\vee}[1];\partial\right]\right)\big/G.\]
The section $s$ also induces the regular function \begin{equation}\label{defreg}
w:\mathcal{E}:=\text{Tot}_X\left(E^{\vee}\right)/G\to\mathbb{A}^1_\mathbb{C}
\end{equation}
defined by
$w(x,v)=\langle s(x), v \rangle$ for $x\in X(\mathbb{C})$ and $v\in E^{\vee}|_x$.
Consider the category of graded matrix factorizations $\text{MF}^{\text{gr}}\left(\mathcal{E}, w\right)$ with respect to the group $\mathbb{C}^*$ mentioned above. There is an equivalence of categories due to Isik \cite{i}, called dimensional reduction or the Köszul equivalence:
\begin{equation}\label{isik}
\text{MF}^{\text{gr}}\left(\mathcal{E}, w\right)\cong D^b(\mathfrak{P}).
\end{equation}
The analogous result in cohomology was proved by Davison \cite{d}.

\subsection{Localization theorems for categories of singularities.}\label{local}
We discuss some properties of categories of singularities on stacks which are used for computations in KHAs.

\begin{prop}\label{injBG}
Let $X$ be an affine scheme with an action of a reductive group $G$ and consider the stack $\X=X/G$. Let $B\subset G$ be a Borel subgroup with maximal torus $T$, and let $\mathcal{Y}:=X/B$ and $\mathcal{Z}=X/T$. There are natural maps $\tau:\mathcal{Z}\to\mathcal{Y}$ and $\pi:\mathcal{Y}\to\X$. Then 
\begin{align*}
\pi^*&: G_0(\X)\hookrightarrow G_0(\mathcal{Y}),\\
\pi^*&: K_0(\X)\hookrightarrow K_0(\mathcal{Y}),\\
\tau^*&: G_0(\mathcal{Y})\cong G_0(\mathcal{Z}),\\
\tau^*&: K_0(\mathcal{Y})\cong K_0(\mathcal{Z}).
\end{align*}
\end{prop}

\begin{proof}
The map $\tau$ is an affine bundle map, so $\tau^*$ is an isomorphism in $K$ and $G$-theory.

Next, we have that $\mathcal{Y}=\left(X\times_B G\right)/G$. The map $\pi^*$ is fully faithful by the projection formula and $\pi_*\mathcal{O}_{\mathcal{Y}}=\mathcal{O}_{\X}$. It has a right adjoint $\pi_*$. Thus the categories
\begin{align*}
    \pi^*&: D^b(\X)\to D^b(\mathcal{Y})\\
    \pi^*&: \text{Perf}(\X)\to \text{Perf}(\mathcal{Y})
\end{align*}
are admissible and the conclusion follows.
\end{proof}

Assume next that $\X=X/T$ for a vector space $X$ with an action of a torus $T$.
Consider a regular function \[f:\X\to\mathbb{A}^1_{\mathbb{C}}\] with $0$ the only critical value. Let $\lambda:\C^*\to T$ be a cocharacter and let $w\in\mathbb{Z}$. Define \[D^b(BT)_{\geq w}\subset D^b(BT)\] the subcategory of complexes on which $\lambda$ acts with weights $\geq w$. It induces a filtration $K_0(BT)_{\geq w}\subset K_0(BT)$. We denote its associated graded pieces by $\text{gr}_w K_0(BT)$.
Let $a: X^\lambda/T\hookrightarrow \X$. It induces a functor $a^*: D_{\text{sg}}(\X_0)\to D_{\text{sg}}\left(X^\lambda_0/T\right)$.
Let
\[D_{\text{sg}}(\X_0)_{\geq w}\subset D_{\text{sg}}(\X_0)\]
be the subcategory of complexes $\F$ such that $\lambda$ acts with weights $\geq w$ on $a^*(F)$. It induces a filtration 
\[K_0\left(D_{\text{sg}}(\X_0)\right)_{\geq w}\subset K_0\left(D_{\text{sg}}(\X_0)\right).\] We denote its associated graded pieces by $\text{gr}_w K_0\left(D_{\text{sg}}(\X_0)\right)$.
$K_0(BT)$ acts on $K_0\left(D_{\text{sg}}(\X_0)\right)$ via the tensor product and respects the above filtrations.

Next, assume that $\mathcal{E}$ is a vector bundle on $\X$ and consider the zero section $\iota:\X\hookrightarrow \mathcal{E}$. Define the Euler class $\text{eu}(\mathcal{E}):=\iota^*\iota_*(1)\in K_0(BT)\cong K_0(\X).$

\begin{prop}\label{zerodiv}
We are continuing in the above framework.
Assume that there exists $\lambda:\C^*\hookrightarrow T$ such that $\mathcal{E}^{\lambda}=\X$. 
Then the class $\text{eu}\,(\mathcal{E})$ is not a zero divisor in $K_0\left(D_{\text{sg}}(\X_0)\right).$
\end{prop}

\begin{proof}
Let $S$ be the set of weights of the normal bundle $N_{\iota}$. We have that
$$\text{eu}(\mathcal{E})=\prod_{\beta\in S}\left(1-q^{\beta}\right)\in K_0(B\C^*).$$ The hypothesis implies that $\langle \lambda, \beta\rangle$ is not zero for $\beta\in S$. Let $v$ be the smallest $\lambda$-weight of a monomial in $\text{eu}(\mathcal{E})$. Then 
\[\text{gr}_v\,\text{eu}(\mathcal{E})=\pm q^v\in \text{gr}_vK_0(BT).\]
Let $w\in\mathbb{Z}$. Multiplication by $\text{eu}(\mathcal{E})$ induces the multiplication by $\pm q^v$-map
\[\text{gr}_w  K_0\left(D_{\text{sg}}(\X_0)\right)\xrightarrow{\sim}
\text{gr}_{v+w} K_0\left(D_{\text{sg}}(\X_0)\right),\] so $\text{eu}(\mathcal{E})$ is not a zero divisor.
\end{proof}

For the next result, let $T$ be a torus, let $X$ be a representation of $T$, and let $Y\hookrightarrow X$ a $T$-equivariant affine subscheme. Denote by $S$ the set of weights $\beta$ of $T$ in $X/X^T$ and by $\mathcal{I}$ the set of functions $1-q^\beta$ with $\beta\in S$. Consider the stack $\X=X/T$. For $M$ a $K_0(BT)$-module, we denote by $M_{\mathcal{I}}$ the localization of $M$ at functions in $\mathcal{I}$.

\begin{thm}\label{loc}
Let $f:\X\to\mathbb{A}^1_{\mathbb{C}}$ be a regular function and let $\lambda:\C^*\to T$ be a cocharacter. Consider the attracting diagram for $\lambda$:
\[\mathcal{Z}:=X^\lambda/T\xleftarrow{q}\mathcal{S}:=X^{\lambda\geq 0}/T
\xrightarrow{p} \X.\] 
Let $\iota:\mathcal{Z}\hookrightarrow\X$ be the natural inclusion map.
There are isomorphisms
\begin{align*}
p_*q^*&: K_0\left(D_{\text{sg}}(\Z_0)\right)_{\mathcal{I}}\xrightarrow{\sim} K_0\left(D_{\text{sg}}(\mathcal{X}_0)\right)_{\mathcal{I}},\\
\iota_*&: K_0\left(D_{\text{sg}}(\mathcal{Z}_0)\right)_{\mathcal{I}}\xrightarrow{\sim} K_0\left(D_{\text{sg}}(\mathcal{X}_0)\right)_{\mathcal{I}},\\
\iota^*&:K_0\left(D_{\text{sg}}(\mathcal{X}_0)\right)_{\mathcal{I}}\xrightarrow{\sim} K_0\left(D_{\text{sg}}(\mathcal{Z}_0)\right)_{\mathcal{I}}.
\end{align*}
\end{thm}

We review Takeda's localization theorem in $K$-theory \cite{tak}:

\begin{prop}\label{takeda}
In the above framework, we have that $G_i^T\left(Y\setminus Y^T\right)_{\mathcal{I}}=0$ for any $i\geq 0$.
\end{prop}

\begin{proof}
This follows from Takeda's original argument \cite[page 79]{tak}. 
Let $\iota:X^T\hookrightarrow X$ be the natural inclusion. 
$G_i^T\left(Y\setminus Y^T\right)$ is a $K_0^T\left(Y\setminus Y^T\right)$-module, so it suffices to show that $K_0^T\left(Y\setminus Y^T\right)_{\mathcal{I}}=0$. There is a restriction map of rings
\[K_0^T\left(X\setminus X^T\right)\to K_0^T\left(Y\setminus Y^T\right).\] It suffices to show that 
$K_0^T\left(X\setminus X^T\right)_{\mathcal{I}}=0$ because then the unit of $K_0^T\left(Y\setminus Y^T\right)_{\mathcal{I}}$ is annihilated and so $K_0^T\left(Y\setminus Y^T\right)_{\mathcal{I}}=0$. It thus suffices to show that 
\[\iota_*: 
K_0^T\left(X^T\right)_{\mathcal{I}}
\xrightarrow{\sim}
K_0^T\left(X\right)_{\mathcal{I}}.\] This is true because $\iota_*$ is multiplication by $\prod_{\beta\in S}\left(1-q^{\beta}\right)$.
\end{proof}

\begin{proof}[Proof of Theorem \ref{loc}]

Let $\mathcal{U}:=\X\setminus\mathcal{S}\subset \X$. Consider the natural inclusion $t:\mathcal{Z}\hookrightarrow \mathcal{S}$. Then $\iota=p\circ t$.
By Proposition \ref{prop1} and the semiorthogonal decomposition \eqref{hl}, there are semiorthogonal decompositions:
\begin{align*}
D_{\text{sg}}(\X_0)&=\big\langle D_{\text{sg}}(\mathcal{Z}_0)_{<w}, \mathbb{D}_w, D_{\text{sg}}(\mathcal{Z}_0)_{\geq w}\big\rangle,\\
D_{\text{sg}}(\mathcal{Z}_0)&=\big\langle D_{\text{sg}}(\mathcal{Z}_0)_{<w}, D_{\text{sg}}(\mathcal{Z}_0)_{\geq w}\big\rangle
\end{align*}
with $\mathbb{D}_w\cong D_{\text{sg}}(\mathcal{U}_0)$ by \eqref{hl}.
By passing to the Grothendieck group, there is a decomposition
$$K_0\left(D_{\text{sg}}(\mathcal{U}_0)\right)\oplus K_0\left(D_{\text{sg}}(\mathcal{Z}_0)\right)\cong K_0\left(D_{\text{sg}}(\mathcal{X}_0)\right),$$ where the map $K_0\left(D_{\text{sg}}(\mathcal{Z}_0)\right)\to K_0\left(D_{\text{sg}}(\mathcal{X}_0)\right)$ is $p_*q^*$. 
We show that \[K_0\left(D_{\text{sg}}(\mathcal{U}_0)\right)_{\mathcal{I}}=0.\] By the definition of the category of singularities, the map
$G_0(\mathcal{U}_0)\twoheadrightarrow K_0\left(D_{\text{sg}}(\mathcal{U}_0)\right)$ is surjective. 
We have that \[G_0\left(\X_0\setminus\mathcal{Z}_0\right)_{\mathcal{I}}\twoheadrightarrow G_0(\mathcal{U}_0)_{\mathcal{I}},\] so by Proposition \ref{takeda} we have that $G_0(\mathcal{U}_0)_{\mathcal{I}}=0.$
Thus 
\begin{equation}\label{takedaiso}
p_*q^*: K_0\left(D_{\text{sg}}(\mathcal{Z}_0)\right)_\mathcal{I}\xrightarrow{\sim} K_0\left(D_{\text{sg}}(\mathcal{X}_0)\right)_\mathcal{I}.
\end{equation}
Let $x\in K_0\left(D_{\text{sg}}(\mathcal{Z}_0)\right)$. Let $e$ be the Euler class of the vector bundle $q:\mathcal{S}\to\mathcal{Z}$.
Then $e$ divides $e':=\prod_{\beta\in S}\left(1-q^{\beta}\right)$.
By the assumption $f|_{\mathcal{S}}=q^*\left(f|_{\mathcal{Z}}\right)$, we have that 
\[t_*(x)=e\cdot q^*(x).\]
The factors of $e$ are in the set $\mathcal{I}$,
so \eqref{takedaiso} implies that \[\iota_*:=p_*t_*: K_0\left(D_{\text{sg}}(\mathcal{Z}_0)\right)_\mathcal{I}\xrightarrow{\sim} K_0\left(D_{\text{sg}}(\mathcal{X}_0)\right)_\mathcal{I}.\]
The last statement follows from $\iota^*\iota_*$ being multiplication by $e'$ and using that the factors of $e'$ are in $\mathcal{I}$. 
\end{proof}

\textbf{Remark.} Via the restriction maps, Proposition \ref{zerodiv} and Theorem \ref{loc} also hold for open substacks of $\X$.

\section{The Hall algebra}\label{3}
\subsection{Definition of the Hall algebra.}\label{multi}
Let $(Q,W)$ be a quiver with potential. 
For $d\in\mathbb{N}^I$, consider the stack of representations \[\X(d)=R(d)/G(d),\] with regular function 
$\text{Tr}(W):\X(d)\to\mathbb{A}^1_{\mathbb{C}},$ see Subsection \ref{quivers} for more details. 
For $d,e\in\mathbb{N}^I$ two dimension vectors, 
consider the stack \[\X(d,e):=R(d,e)/G(d,e)\] of pairs of representations $ A\subset B$, where $A$ has dimension $d$ and $B$ has dimension $d+e$. 
Let $\theta\in\mathbb{Q}^I$ be a King stability condition. Define the slope function: $$\tau(d):=\frac{\sum_{i\in I}\theta^i d^i}{\sum_{i\in I} d^i}.$$
For a fixed slope $\mu$, let $\Lambda_{\mu}\subset \mathbb{N}^I$ be the monoid of dimension vectors with slope $\mu$. We denote by $\X(d)^{\text{ss}}\subset \X(d)$ the substack of $\theta$-semistable representations.
There is a cocharacter $\lambda_{d,e}$ whose diagram of fixed and attracting loci \eqref{pq} is
\begin{equation}\label{ppqq}
    \X(d)^{\text{ss}}\times \X(e)^{\text{ss}}\xleftarrow{q_{d,e}} \X(d,e)^{\text{ss}} \xrightarrow{p_{d,e}} \X(d+e)^{\text{ss}}.
    \end{equation}
Fix such a cocharacter $\lambda_{d,e}$.
The induced regular functions are compatible with respect to these maps: 
$$p_{d,e}^*\text{Tr}(W_{d+e})=q_{d,e}^*\left(\text{Tr}(W_d)+\text{Tr}(W_e)\right).$$ 
We use $p$ and $q$ instead of $p_{d,e}$ and $q_{d,e}$ when there is no danger of confusion.

For every edge $e\in E$, let $\mathbb{C}^*$ act on $\text{Hom}\,\left(\C^{s(e)},\C^{t(e)}\right)$ by scalar multiplication. We denote the product of these multiplicative groups by $\left(\mathbb{C}^*\right)^E$.
The Hall algebras considered in this paper are equivariant with respect to a torus $T$ such that $T\subset \left(\mathbb{C}^*\right)^E$ and $W$ is $T$-invariant. We say that $(Q,W)$ satisfies \textbf{Assumption A} if there exists an extra $\mathbb{C}^*\subset \left(\mathbb{C}^*\right)^E$
 such that the regular functions $\text{Tr}(W_d)$ are all homogeneous of weight $2$. We consider graded categories of matrix factorizations with respect to such a fixed $\mathbb{C}^*$. Different choices of such $\mathbb{C}^*$ give different categories $\text{MF}^{\text{gr}}$, but all these categories have the same Grothendieck group \cite[Corollary 3.13]{T2}.

Consider the diagonal map
$\delta: BT\to BT\times BT.$
There are induced maps 
\[ \delta: \left(\X(d)^{\text{ss}}\times\X(e)^{\text{ss}}\right)/T
\to
\left(\X(d)^{\text{ss}}/T\right)\times\left(\X(e)^{\text{ss}}/T\right).\]

\begin{defn}\label{multiplication}
Consider the functor
 \begin{multline*}
     m_{d,e}:=p_*q^*\,\text{TS}\,\delta^*:
 \text{MF}_T\left(\X(d)^{\text{ss}}, W_d\right)\,\boxtimes\, \text{MF}_T\left(\X(e)^{\text{ss}}, W_e\right) \\\to
 \text{MF}_T\left(\X(d+e)^{\text{ss}}, W_{d+e}\right),
 \end{multline*}
where $\text{TS}$ is the Thom-Sebastiani functor, see Subsection \ref{TSiso}. 
Under Assumption A, we consider 
the functor
\begin{multline*}
m_{d,e}:=p_*q^*\,\text{TS}\,\delta^*:
 \text{MF}^{\text{gr}}_T\left(\X(d)^{\text{ss}}, W_d\right)\,\boxtimes\, \text{MF}^{\text{gr}}_T\left(\X(e)^{\text{ss}}, W_e\right) \\\rightarrow 
 \text{MF}^{\text{gr}}_T\left(\X(d+e)^{\text{ss}}, W_{d+e}\right).\end{multline*}

\end{defn}

\begin{defn}
Consider the $\Lambda_{\mu}$-graded category
\[\text{HA}_T(Q,W)_{\mu}:=\bigoplus_{d\in\Lambda_\mu} \text{MF}_T\left(\X(d)^{\text{ss}}, W_d\right).\] Under Assumption A, we consider the $\Lambda_{\mu}$-graded category
\[\text{HA}^{\text{gr}}_T(Q,W)_{\mu}:=\bigoplus_{d\in\Lambda_\mu} \text{MF}^{\text{gr}}_T\left(\X(d)^{\text{ss}}, W_d\right).\] 
We call these categories the \textit{categorical Hall algebras} of $(Q,W)$. We call the Grothendieck group of these categories the \textit{K-theoretic Hall algebras} of $(Q,W)$. 
\end{defn}

In this section, we prove Theorem \ref{thm:1} and its version for graded matrix factorizations:

\begin{thm}\label{thm1}
The categories $\text{HA}_T(Q,W)_{\mu}$ and $\text{HA}^{\text{gr}}_T(Q,W)_{\mu}$ are monoidal with respect to the multiplication functors $m$.
\end{thm}

\begin{proof}
We discuss the statement for $\text{MF}$, the one for $\text{MF}^{\text{gr}}$ follows in the same way.
Let $d, e, f\in \Lambda_\mu$. Let $\X(d,e,f)$ be the stacks of triples of representations of $Q$ \[A\subset B\subset C\] with $A$ of dimension $d$, $B/A$ of dimension $e$, and $C/B$ of dimension $f$.
We use the shorthand notations
\begin{align*}
    \textbf{C}_T(d)&:=\text{MF}\left(\X(d)^{\text{ss}}, W_d\right),\\
    \textbf{C}_T(d,e)&:=\text{MF}\left(\X(d,e)^{\text{ss}}, W_{d+e}\right),\\
    \textbf{C}_T(d,e,f)&:=\text{MF}\left(\X(d,e,f)^{\text{ss}}, W_{d+e+f}\right)\,\text{etc.}
\end{align*}
We also use the shorthand notation $\X_T(d)=\X(d)/T$.
We need to show that the following diagram commutes:
\begin{equation}\label{diagr1}
\begin{tikzcd}
  \textbf{C}_T(d)\boxtimes \textbf{C}_T(e)\boxtimes \textbf{C}_T(f) \arrow[r, "q_{e,f}^* \delta^*"] \arrow[d, "q_{d,e}^*\delta^*"]
    & \textbf{C}_T(d)\boxtimes \textbf{C}_T(e,f) \arrow[d, "q_1^*\delta^*"] \arrow[r, "p_{e,f*}"] & \textbf{C}_T(d)\boxtimes \textbf{C}_T(e+f) \arrow[d, "q_{d,e+f}^*\delta^*"] \\
  \textbf{C}_T(d,e)\boxtimes \textbf{C}_T(f) \arrow[d, "p_{d,e*}"] \arrow[r, "q_2^*\delta^*"] 
& \textbf{C}_T(d,e,f) \arrow[d, "p_{1*}"] \arrow[r, "p_{2*}"] & 
\textbf{C}_T(d, e+f) \arrow[d, "p_{d,e+f*}"] \\
\textbf{C}_T(d+e)\boxtimes \textbf{C}_T(f) \arrow[r, "q_{d+e,f}^*\delta^*"] & \textbf{C}_T(d+e,f) \arrow[r, "p_{d+e,f*}"] & \textbf{C}_T(d+e+f),
\end{tikzcd}
\end{equation}
where we abused notation and dropped $\text{TS}$ from the notation of the morphisms. The Thom-Sebastiani functor commutes with the pullbacks and pushforwards above. 
For the upper right corner, the maps are induced from the ones of the cartesian diagram \begin{equation*}
\begin{tikzcd}
  \X_T(d,e,f) \arrow[r,"p_2"] \arrow[d,"\delta q_1"]
    & \X_T(d,e+f) \arrow[d,"\delta q_{d,e+f}"] \\
  \X_T(d)\times\X_T(e,f) \arrow[r,"p_{e,f}"]
&\X_T(d)\times\X_T(e+f),
\end{tikzcd}
\end{equation*}
and one can use proper base change to deduce that $p_{2*}q_1^*\delta^*=q_{d,e+f}^*\delta^* p_{e,f*}$. A similar argument shows that the lower left corner commutes.
The lower right corner clearly commutes.
For the upper left corner of diagram \eqref{diagr1}, consider the diagram
\begin{equation}\label{diagr}
\begin{tikzcd}
\X_T(d,e,f)\arrow[d,"q_2"]\arrow[r,"q_1"]&(\X(d)\times\X(e,f))/T\arrow[d,"q_{e,f}"]\arrow[r,"\delta"]&\X_T(d)\times\X_T(e,f)\arrow[d,"q_{e,f}"]\\
(\X(d,e)\times\X(f))/T\arrow[d,"\delta"]\arrow[r,"q_{d,e}"]& (\X(d)\times\X(e)\times\X(f))/T\arrow[r,"\delta"]\arrow[d,"\delta"]& \X_T(d)\times (\X(e)\times\X(f))/T\arrow[d,"\text{id}\,\boxtimes\,\delta"]\\
\X_T(d,e)\times\X_T(f)\arrow[r,"q_{d,e}"]& (\X(d)\times\X(e))/T\times\X_T(f)\arrow[r,"\delta\,\boxtimes\,\text{id}"]&\X_T(d)\times\X_T(e)\times\X_T(f),
\end{tikzcd}
\end{equation}
where we have slightly abused notation involving the maps above.
The upper left corner of \eqref{diagr} clearly commutes. The upper right and lower left corners are base change diagrams. The lower right corner of \eqref{diagr} commutes because all maps are base-change of the map 
\[\X_T(d)\times\X_T(e)\times\X_T(f)\to BT\times BT\times BT\]
along the maps in the 
cartesian diagram
\begin{equation*}
\begin{tikzcd}
  BT \arrow[r, "\delta"] \arrow[d, "\delta"]
    & BT\times BT \arrow[d, "\delta\times 1"] \\
  BT\times BT \arrow[r, "1\times\delta"]
& BT \times BT\times BT.
\end{tikzcd}
\end{equation*}
The diagram \eqref{diagr} commutes and thus diagram \eqref{diagr1} commutes as well.
\end{proof}

\subsection{Comparison with the preprojective KHA}\label{compHA}

\subsubsection{}\label{Koszuls}
Consider a quiver $\widetilde{Q}=\left(I, \widetilde{E}\right)$ and a decomposition of sets  $\widetilde{E}=E^d\sqcup C$. 
Let $Q^d=(I,E^d)$ and $Q'=(I,C)$.
The group $\mathbb{C}^*$ acts on representations of $\widetilde{Q}$ by scaling the linear maps corresponding to edges in $C$ with weight $2$. 
Consider a potential $\widetilde{W}$ of $\widetilde{Q}$ on which 
$\mathbb{C}^*$ acts with weight $2$. The set $C$ is called a \textit{cut} for $\left(\widetilde{Q}, \widetilde{W}\right)$ in the literature. 
Denote by $\widetilde{\X(a)}$ the moduli stack of representations of dimension $a$ for the quiver $\widetilde{Q}$ and by $\X^d(a)=R^d(a)/G(a)$ the analogous stack for the quiver $Q^d$. We consider the category of graded matrix factorizations $\text{MF}^{\text{gr}}\left(\widetilde{\X(a)}, \widetilde{W}_a\right)$ with respect to the action of the group $\mathbb{C}^*$ mentioned above.
Denote the representation space of $Q'$ by 
$C(a)$. 
We abuse notation and denote by $C(a)$ the natural vector bundle on $\X^d(a)$. 
Write \[\widetilde{W}=\sum_{c\in C}cW_c,\] where $W_c$ is a path of $Q^d$. Define the algebra \[P:=\mathbb{C}\left[Q^d\right]\Big/\mathcal{J},\] where $\mathcal{J}$ is the two-sided ideal generated by $W_c$ for $c\in C$.
The potential $\widetilde{W}$ induces a section $s\in \Gamma\left(\X^d(a), C(a)^{\vee}\right)$, and thus a map $\partial:C(a)\to \mathcal{O}_{\X^d(a)}$. The corresponding regular function constructed in \eqref{defreg} is 
\[\text{Tr}\,\widetilde{W}:\widetilde{\X(a)}\to\mathbb{A}^1_{\mathbb{C}}.\]
The moduli stack of representations of $P$ of dimension $a$ is the Köszul stack
\begin{equation}\label{der}
\mathfrak{P}(a):=\text{Spec}\left(\mathcal{O}_{R^d(a)}\left[C(a)[1];\partial\right]\right)\big/ G(a).
\end{equation}
By the dimensional reduction equivalence \eqref{isik}, see also \cite[Appendix A.3]{d} for the argument in cohomology, we have an equivalence:
\begin{equation}\label{Koszul}
\Phi: D^b\left(\mathfrak{P}(a)\right)\cong \text{MF}^{\text{gr}}\left(\widetilde{\X(a)}, \widetilde{W}\right).
\end{equation}

\subsubsection{}
Let $Q=(I,E)$ be a quiver and consider the tripled quiver $\left(\widetilde{Q}, \widetilde{W}\right)$ from Subsection \ref{tripledef}. $Q^d$ is the doubled quiver of $Q$.
Let $C=\{\omega_i|\,i\in I\}$. Then $Q^d$ is the doubled quiver of $Q$.
In this case, the ideal $\mathcal{J}$ is \[\mathcal{J}=\left(\sum_{\substack{e\in E,\\ s(e)=i}}[\overline{e}, e], i\in I\right),\] so the algebra $P$ is the preprojective algebra of $Q$ and thus $\mathfrak{P}(a)$ is the stack of representation of $P$ of dimension $a$. Consider the categorical preprojective Hall algebra studied by Varagnolo--Vasserot \cite{VV}:
\[\text{HA}_T(Q):=\bigoplus_{d\in\mathbb{N}^I}D^b_T\left(\mathfrak{P}(d)\right).\]
We obtain an equivalence of Hall algebra categories
\[\text{HA}_T(Q)\cong \text{HA}_T\left(\widetilde{Q}, \widetilde{W}\right).\]

\subsubsection{}\label{comparisontwoKHAs}
Let $a,b,d\in\mathbb{N}^I$ with $d=a+b$. We denote the multiplication maps for $(a,b)$ for $\widetilde{Q}$ by $\widetilde{p}$, $\widetilde{q}$ and for $Q^d$ by $p$, $q$. Let $\lambda$ be a cocharacter $\lambda_{a,b}$ as in \eqref{ppqq}. 
Define the line bundle on $\widetilde{\X(a)}\times\widetilde{\X(b)}$:
\begin{equation}\label{omega}
\omega_{a,b}:=\det \Big(C(d)^{\lambda\leq 0}\big/C(d)^{\lambda}\Big)
\end{equation}
We twist the multiplication $\widetilde{m}$ on $\text{HA}_T\left(\widetilde{Q}, \widetilde{W}\right)$ and define
\begin{multline*}
\widetilde{m'}:=\widetilde{p}_*\widetilde{q}^*\left(-\otimes\omega_{a,b}\right)\,\text{TS}\,\delta^*: \text{MF}^{\text{gr}}_T\left(\widetilde{\X(a)}, \widetilde{W_a}\right)\otimes \text{MF}^{\text{gr}}_T\left(\widetilde{\X(b)}, \widetilde{W_b}\right)\\
\to 
\text{MF}^{\text{gr}}_T\left(\widetilde{\X(d)}, \widetilde{W_d}\right).
\end{multline*}
Let $m$ be the multiplication of the preprojective $\text{HA}_T(Q)$. By a $T$-equivariant version of \cite[Proposition 2.2]{P3}, 
the following diagram commutes:
\begin{equation}\label{compKHA}
    \begin{tikzcd}
    D^b_T\left(\mathfrak{P}(a)\right)\otimes  D^b_T\left(\mathfrak{P}(b)\right) 
    \arrow[d, "\Phi\,\otimes\,\Phi"]\arrow[r,"m"]& D^b_T\left(\mathfrak{P}(d)\right)\arrow[d,"\Phi"]\\
    \text{MF}^{\text{gr}}_T\left(\widetilde{\X(a)}, \widetilde{W_a}\right)\otimes \text{MF}^{\text{gr}}_T\left(\widetilde{\X(b)}, \widetilde{W_b}\right)
    \arrow[r,"\widetilde{m'}"]&
    \text{MF}^{\text{gr}}_T\left(\widetilde{\X(d)}, \widetilde{W_d}\right).
    \end{tikzcd}
\end{equation}
When passing to $K_0$, the multiplication $\widetilde{m'}$ is conjugation of $\widetilde{m}$ by an explicit rational function, see \cite[Subsection 2.3.7]{VV}.

\subsection{Examples of KHA}\label{6}

\subsubsection{The potential zero case}
Let $Q$ be an arbitrary quiver.
Let $i,i'$ be vertices of $Q$ and let $\{1,\cdots, \varepsilon(i,i')\}$ be the set of edges from $i$ to $i'$. Consider the action of $\left(\mathbb{C}^*\right)^{\varepsilon(i,i')}$ on $R(d)$ whose $j$th copy acts on $R(d)$ with weight $1$ on the factor $\text{Hom}\left(\mathbb{C}^{d_{s(j)}}, \mathbb{C}^{d_{t(j)}}\right)$ corresponding to the edge $j$. Denote by $q_j$ the weight corresponding to the $j$th copy of $\mathbb{C}^*$. 
Define
\begin{equation}\label{zeta}
    \zeta_{ii'}(z):= \frac{\left(1-q_1^{-1}z^{-1}\right)\cdots \left(1-q_{\varepsilon(i,i')}^{-1}z^{-1}\right)}{\left(1-z^{-1}\right)^{\delta_{ii'}}}, \end{equation}
where $\delta_{ii'}$ is $1$ if $i=i'$ and $0$ otherwise. 
Let $T$ be a subtorus of $\left(\mathbb{C}^*\right)^E$ which fixes $W$. We also use the notation $q_j$ for the corresponding weight of $T$. 
For $d\in\mathbb{N}^I$, let $\mathfrak{S}_d:=\times_{i\in I}\,\mathfrak{S}_{d_i}$ be the Weyl group of $G(d)$. 

\begin{prop}\label{shuffle}
The $\mathbb{N}^I$-graded vector space of $\text{KHA}_T(Q,0)$ has $d$ graded space 
\[K_0^T(\X(d))\cong K_0(BT)\left[z_{i,j}^{\pm 1}\right]^{\mathfrak{S}_d},\]
where $i\in I$ and $1\leq j\leq d_i$. 
Let $f\in K_0^T(\X(a))$, $g\in K_0^T(\X(b))$ with $a+b=d$. Then the multiplication in $\text{KHA}_T(Q,0)$ is
\[f\cdot g=\sum_{w\in \mathfrak{S}_{d}/\mathfrak{S}_a\times\mathfrak{S}_b} w\left(fg\prod_{\substack{i,i'\in I\\ j\leq a_i\\ j'>a_{i'}}}\zeta_{ii'}\left(\frac{z_{ij}}{z_{i'j'}}\right)
\right).\]
\end{prop}

\begin{proof}
Consider the maps \begin{align*}
    \iota_{a,b}&:R(a,b)/G(a,b)\hookrightarrow R(d)/G(a,b)\\
\pi_{a,b}&:R(d)/G(a,b)\to \X(d).
\end{align*}
The multiplication is defined by the composition
\[\X(a)\times\X(b)\xrightarrow{q_{a,b}^*} \X(a,b)\xrightarrow{\iota_{a,b*}} R(d)/G(a,b)\xrightarrow{\pi_{a,b*}} \X(d).\]
The pullback map is an isomorphism \begin{equation}\label{qiso}
    q_{a,b}^*:K_0^T(\X(a)\times\X(b))\cong K_0^T(\X(a,b)).
    \end{equation}
    Let $N$ be the normal bundle of the map $\iota_{a,b}$.
The weights of $N^{\vee}$ are 
$q^{-1}_ez_{ij}^{-1}z_{i'j'}$, where $j\leq a_i$, $j'>a_{i'}$, and $e\in\{1,\cdots,\varepsilon(i,i')\}$ is an edge between $i$ and $i'$. 
For $h\in K_0^T(\X(a,b))$, we thus have that
\begin{equation}\label{1}
\iota_{a,b*}(h)=h\prod_{\substack{i,i'\in I\\ j\leq a_i\\ j'>a_{i'}}}\left(1-q_1^{-1} z_{ij}^{-1}z_{i'j'}\right)\cdots \left(1-q_{\varepsilon(i,i')}^{-1} z_{ij}^{-1}z_{i'j'}\right).
\end{equation}
Further, for $h\in K_0^T(R(d)/G(a,b))$, we have that
\begin{equation}\label{22}
\pi_{a,b*}(h)=\sum_{w\in\mathfrak{S}_{d}/\mathfrak{S}_a\times\mathfrak{S}_b} w\left(\frac{h}{\prod_{\substack{i\in I\\ j\leq a_i<k}}\left(1-z_{ij}^{-1}z_{ik}\right)}\right),
\end{equation}
see, for example, \cite[Proposition 1.2 or the proof of Proposition 2.3]{yz}.
The formula for the multiplication of the $\text{KHA}_T(Q,0)$ follows from \eqref{qiso}, \eqref{1}, and \eqref{22}.
\end{proof}

A shuffle formula for the product of $\text{CoHA}$ appears in \cite[Section 2.4]{ks} and for a general oriented cohomological theory in \cite{yz}.

Proposition \ref{shuffle} implies that:
\begin{cor}
Let $J$ be Jordan quiver. Consider the action of $\mathbb{C}^*$ which scales representations of $J$ with weight $1$.
There is an isomorphism $$\text{KHA}_{\C^*}(J,0)\cong U_q^{>}(L\mathfrak{sl}_2).$$
\end{cor}

\begin{proof}

The dimension $d$ graded component of $\text{KHA}_{\C^*}(J,0)$ is $$K_0^{\C^*}(\X(d))= \mathbb{Z}\left[q^{\pm 1}\right]\left[z_1^{\pm 1},\cdots, z_d^{\pm 1}\right]^{\mathfrak{S}_d}.$$ In this case, 
\[\zeta(z)=\frac{1-q^{-1}z^{-1}}{1-z^{-1}}.\] Consider dimension vectors $a,b,d\in\mathbb{N}$ such that $a+b=d$.
The multiplication of $f\in K_0^{\C^*}(\X(a))$ and $g\in K_0^{\C^*}(\X(b))$ is
$$(f\cdot g)(z_1,\cdots, z_{d})=\sum_{\mathfrak{S}_d/\mathfrak{S}_a\times\mathfrak{S}_b}
\left(f\left(z_1,\cdots, z_a\right)g\left(z_{a+1},\cdots, z_d\right)\prod_{\substack{1\leq i\leq a,\\
a+1\leq j\leq d}} \zeta\left(\frac{z_i}{z_j}\right)\right).$$ By \cite[Theorem 3.5]{ts}, the quantum group $U_q^{>}(L\mathfrak{sl}_2)$ has the same shuffle product description.
\end{proof}

\subsubsection{}
Let $(Q,W)$ be an arbitrary quiver with potential and let $T$ be a torus preserving the potential $W$. 

\begin{prop}\label{defo}
Assume that $R(d)^{T(d)\times T}$ is in the zero locus of $\text{Tr}\left(W_d\right)$.
The inclusion $\iota_d:\X(d)_0\to \X(d)$ induces an algebra morphism 
$$\iota_{d*}: \text{KHA}_T(Q,W)\to \text{KHA}_T(Q,0).$$
\end{prop}

\begin{proof}
First, the morphism $\iota_{d*}: D^b_T(\X(d)_0)\to D^b_T(\X(d))$ commutes with the maps used in the definition of multiplication. The attracting maps for the category of singularities are induced from
\[(\X(a)\times\X(b))_0\xleftarrow{q}\X(a,b)_0\xrightarrow{p}\X(d)_0.\]
Let $a,b,d\in\mathbb{N}^I$ with $d=a+b$.
For this, we need to check that the following diagram is commutative:
\begin{equation}\label{diagiota}
\begin{tikzcd}
D^b_T(\X(a)_0)\boxtimes D^b_T(\X(b)_0) \arrow{d}{\iota_{a*}\boxtimes \iota_{b*}}\arrow{r}{i_*p_*q^*} & D^b_T(\X(d)_0) \arrow{d}{\iota_{d*}}\\
D^b_T(\X(a))\boxtimes D^b_T(\X(b)) \arrow{r}{p_*q^*} & D^b_T(\X(d)),
\end{tikzcd}
\end{equation}
where $i:\X(a)_0\times \X(b)_0\to \X(d)_0$. Recall that $i_*$ is the Thom-Sebastiani functor.
It is enough to show that the following diagram commutes:
\begin{equation*}
\begin{tikzcd}
D^b_T\big((\X(a)\times\X(b))_0\big) \arrow{d}{\iota_{d*}} \arrow{r}{q^*} & D^b_T(\X(a,b)_0) \arrow{r}{p_*} \arrow{d}{\iota_{d*}} & D^b_T(\X(d)_0) \arrow{d}{\iota_{d*}}\\
D^b_T\big(\X(a)\times\X(b)\big) \arrow{r}{q^*} & D^b_T(\X(a,b)) \arrow{r}{p_*} & D^b_T(\X(d)).
\end{tikzcd}
\end{equation*}
The left corner commutes from proper base change. The right corner clearly commutes.

We next show that $\iota_{d*}\, K^T_0(\X(d)_0)=0$. 
By Propositions \ref{injBG} and \ref{loc}, the restriction map is an isomorphism
\[K_0^T(\X(d))_{\mathcal{I}}\xrightarrow{\sim} K_0^{T(d)\times T}\left(R(d)^{T(d)\times T}\right)_{\mathcal{I}},\]
where ${\mathcal{I}}$ is the set in Proposition \ref{loc}. 
The map $K_0^T(\X(d))\hookrightarrow K_0^T(\X(d))_{\mathcal{I}}$ is injective. By the assumption on $(Q,W)$ and $T$, the map $R(d)^{T\times T(d)}\big/T(d)\hookrightarrow \X(d)$ factors through
\[R(d)^{T\times T(d)}\big/T(d)\hookrightarrow \X(d)_0\hookrightarrow \X(d).\]
The following map is thus injective
\[\iota_d^*: K_0^T(\X(d))\hookrightarrow K_0^T(\X(d)_0).\] For any complex $F$ in $\text{Perf}\,(\X(d)_0)$, we have an exact triangle  \[F[1]\to\iota_{d}^*\iota_{d*}(F)\to F\xrightarrow{[1]},\]
see \cite[Remark 1.8]{hlp}, so $\iota_d^*\iota_{d*}(F)=0$ in $K_0^T(\X(d)_0)$. 
Thus the map \[\iota_{d*}: G^T_0(\X(d)_0)\to K^T_0(\X(d))\] factors through $$K^T_0\left(D_{\text{sg}}(\X(d)_0)\right)\to K_0^T(\X(d)).$$ 
The induced maps $K^T_0\left(D_{\text{sg}}(\X(d)_0)\right)\to K^T_0(\X(d))$ respect multiplication by \eqref{diagiota}, so
$$\iota_{d*}: \text{KHA}_T(Q,W)\to \text{KHA}_T(Q,0)$$ is an algebra morphism.

\end{proof}

\textbf{Remark.} Let $Q$ be a quiver. Consider the tripled quiver $\left(\widetilde{Q}, \widetilde{W}\right)$. Using \eqref{isogra}, \eqref{compKHA}, and Proposition \ref{defo}, we obtain an algebra morphism from the preprojective KHA of $Q$ to a shuffle algebra
\[\text{KHA}_T(Q)\cong \text{KHA}^{\text{gr}}_T\left(\widetilde{Q}, \widetilde{W}\right)\cong \text{KHA}_T\left(\widetilde{Q}, \widetilde{W}\right)
\to \text{KHA}_T\left(\widetilde{Q}, 0\right).\]
This allows for explicit computations in preprojective KHAs in conjunction with generation results and allows for checking Conjecture \ref{conjecture} in particular cases. The most general result in this direction is due to Varagnolo--Vasserot \cite{VV}, who checked the conjecture for finite and affine type quiver except $A_1^{(1)}$.

\section{Representations of KHA}\label{7}

\subsection{} 
Let $(Q',W')$ be a quiver with potential, let $T$ be a torus as in Section \ref{multi}, and let $\theta'\in \mathbb{Q}^{I'}$ be a stability condition for $Q'$. Denote by $\tau$ the slope function. 
Let $Q\subset Q'$ be a subquiver with $I'\setminus I=\{ \infty \}$ and denote by
\begin{align*}
    W&:=W'|_{Q},\\
    \theta&:=\theta'|_Q.
\end{align*}
Fix $\mu$ a slope for $(Q, \theta)$. Let $d\in\Lambda_{\mu}$ be a dimension vector of $Q$. 
For any dimension vector $d\in \mathbb{N}^I$, consider the dimension vectors \begin{align*}
   \widetilde{d}&:=(1,d),\\
   d^o&:=(0,d)
\end{align*} of $Q'$, where we identify $\mathbb{N}\times \mathbb{N}^I\cong \mathbb{N}^{I'}$. We denote by $\X$ the moduli stacks for $Q'$. Then $\X\left(d^o\right)$ is the moduli stack of representations of $Q$. We say that $\left(Q', Q, \theta', \mu\right)$ satisfies \textbf{Assumption B} if 
\[\theta'(\infty)=\mu+\varepsilon \] for $0<\varepsilon\ll 1$. Then, for any $d\in\Lambda_{\mu}$, we have that
\[\mu=\tau\left(d^o\right)<\tau\left(\widetilde{d}\right)=\mu+\varepsilon'\] for $0<\varepsilon'\ll 1$.
Let $d,e\in\Lambda_{\mu}$.
Denote by $\X\left(d^o, \widetilde{e}\right)^{\text{ss}}$ the stack of pairs $A\subset B$ such that $A$ is $\theta$-semistable of dimension $d^o$ and $B$ is $\theta'$-semistable of dimension $\widetilde{d+e}$. Then
$B/A$ is $\theta'$-semistable of dimension $\widetilde{e}$ and $A$ is $\theta'$-semistable. 
There are maps
\begin{align*}
    t_{d,e}&:\X\left(d^o, \widetilde{e}\right)^{\text{ss}}\to \X\left( d^o\right)^{\text{ss}}\times \X\left( \widetilde{e}\right)^{\text{ss}},\\
    s_{d,e}&:\X\left(d^o, \widetilde{e}\right)^{\text{ss}}\to \X\left( \widetilde{d+e}\right)^{\text{ss}}.
    \end{align*}
    The map $t_{d,e}$ is an affine bundle map and the map $s_{d,e}$ is proper, see also the discussion in \cite[Subsection 4.1]{s}. 
    Denote by     
\[\text{R}_T\left(Q',W'\right)_{\mu}:=\bigoplus_{d\in\Lambda_\mu} \text{MF}^{\text{gr}}_T\left(\X\left(\widetilde{d}\right), W\right).\]    
We denote its Grothendieck group by $\text{KR}_T\left(Q',W'\right)_{\mu}$.

\begin{prop}\label{repres}
In the above framework and under Assumption B, $\text{KHA}_T(Q,W)_\mu$ naturally acts on $\text{KR}_T\left(Q',W'\right)_{\mu}$ via the functors
\begin{multline}
p_{d,e*}q_{d,e}^*\,\text{TS}\,\delta: \text{MF}_T\left(
\X\left(d^o\right)^{\text{ss}}, W\right)\otimes \text{MF}_T\left(\X\left(\widetilde{e}\right)^{\text{ss}}, W\right)\\\to
\text{MF}_T\left(\X\left(\widetilde{d+e}\right)^{\text{ss}}, W\right)\end{multline}
for $d,e\in\Lambda_{\mu}$. If Assumption A is satisfied, we can consider graded categories $\text{MF}^{\text{gr}}$ and the analogous statement holds in that case as well.
\end{prop}

\begin{proof}
Same proof as in Theorem \ref{thm1} works here.
\end{proof}

\subsection{} 

Let $Q_3$ be the quiver with one vertex $0$ and three loops $x,y,z,$ consider the potential $W_3=xyz-xzy$, and consider the zero stability condition.
Let $Q_3^f=\left(I^f, E^f\right)$ be the quiver with $I^f=\{0, \infty\}$ and $E^f=\{x,y,z,t\}$. Consider the stability condition $\theta$ of $Q^f$ from Subsection \ref{framed}. For $d\in\mathbb{N}$, denote by $\widetilde{d}=(1,d)$ a dimension vector of $Q^f$. Assumption B is satisfied in this case by the construction of $\theta$. Let $T\subset \left(\mathbb{C}^*\right)^3$ be a torus which fixes $W$.
Consider the regular function
\[\text{Tr}\left(W_{3,\widetilde{d}}\right): \X\left(\widetilde{d}\right)^{\text{ss}}\to\mathbb{A}^1_{\mathbb{C}}.\]
Its critical locus is $\text{Hilb}\left(\C^3,d\right)$. Use the notation \[K_{\text{crit}, T}\left(\text{Hilb}\left(\mathbb{A}^3_{\mathbb{C}}, d\right)\right):=K_0\left(\text{MF}_T\left(\X\left(\widetilde{d}\right)^{\text{ss}}, W_{3,\widetilde{d}}\right)\right).\]
Proposition \ref{repres} constructs an action of $\text{KHA}_T\left(Q_3,W_3\right)$ on $$\bigoplus_{d\geq 0}K_{\text{crit}, T}\left(\text{Hilb}\left(\mathbb{A}^3_{\mathbb{C}}, d\right)\right).$$

\subsection{} 
We explain how Proposition \ref{repres} can be used to constructs representations on $K$-theory of Nakajima quiver varieties following the analogous construction in cohomology \cite[Section 6.3.]{d2}. 
Given a quiver $Q=(I,E)$ and $f\in\mathbb{N}^I$, denote by $Q^f$ the framed quiver with vertices $I\cup\{\infty\}$ and $f_i$ new edges from $\infty$ to $i\in I$. 
Recall from Subsection \ref{tripledef} the construction of the double quiver $\overline{Q}$ and tripled quiver with potential $\left(\widetilde{Q}, \widetilde{W}\right)$. 
Let $\theta=0$ be the zero stability condition for $Q$. Consider the stability condition $\theta^f$
defined in Subsection \ref{framed} for the quiver $Q^f$. 

Define the torus $T'\cong \left(\mathbb{C}^*\right)^E$, where for $e\in E$ the corresponding $\mathbb{C}^*$ acts with weight $1$ on the linear map corresponding to $e$, with weight $-1$ on the linear map corresponding to $\overline{e}$, and with weight $0$ on $\omega_i$ for any $i\in I$. We consider a torus $T\subset T'$.

For $d\in\mathbb{N}^I$, denote by $\widetilde{d}=(1,d)\in\mathbb{N}\times \mathbb{N}^I$. Consider the map that forgets the action of $\omega_i$: 
$$\pi: \X\left(\widetilde{Q^f},\widetilde{d}\right)\to \X\left(\overline{Q^f},\widetilde{d}\right)$$ the map that forgets the action of $\omega_i$. Observe that the doubled and tripled quivers considered above are for $Q^f$.
Ben Davison showed in \cite[Lemma 6.5.]{d2} that the inclusion
\[\pi^{-1}\left(\X\left(\overline{Q^f},\widetilde{d}\right)^{\text{ss}}\right)\hookrightarrow \X\left(\widetilde{Q^f},\widetilde{d}\right)^{\text{ss}}\] induces an equality
\begin{equation}\label{davi}
\pi^{-1}\left(\X\left(\overline{Q^f},\widetilde{d}\right)^{\text{ss}}\right)\cap \text{crit}\,\left(\text{Tr}\,\widetilde{W^f}\right)=
\X\left(\widetilde{Q^f},\widetilde{d}\right)^{\text{ss}}\cap \text{crit}\,\left(\text{Tr}\,\widetilde{W^f}\right).\end{equation}
We consider the extra $\mathbb{C}^*$-action induced by acting with weight $2$ on the linear maps corresponding to $\omega_i$ for $i\in I$ and with weight $0$ on the other linear maps. We consider $\text{MF}^{\text{gr}}$ with respect to this $\mathbb{C}^*$-action.
The equality \eqref{davi} implies that
$$\text{MF}^{\text{gr}}\left(\X\left(\widetilde{Q^f},\widetilde{d}\right)^{\text{ss}}, W\right)\cong \text{MF}^{\text{gr}}\left(\pi^{-1}\left(\X\left(\overline{Q^f},\widetilde{d}\right)^{\text{ss}}\right), W\right).$$
Consider the stack 
$\mathcal{Y}:=\X\left(\overline{Q^f},\widetilde{d}\right)^{\text{ss}}$, the $\mathfrak{g}\left(\widetilde{d}\right)$-vector bundle $\pi^{-1}\left(\mathcal{Y}\right)$, and the regular function $\text{Tr}\,W_{\widetilde{d}}$. We also denote the vector bundle by $\mathfrak{g}\left(\widetilde{d}\right)$. The potential is constructed as in Subsection \ref{dimred0} via the section $s\in \Gamma\left(\mathcal{Y}, \mathfrak{g}(d)\right)$ corresponding to the natural moment map
\[s: T^*R\left(\widetilde{d}\right)\cong 
\overline{R\left(\widetilde{d}\right)}\to \mathfrak{g}\left(\widetilde{d}\right).\] It induces a map $\partial: \mathfrak{g}\left(\widetilde{d}\right)\to \mathcal{O}_{\overline{R\left(\widetilde{d}\right)}}$.
The diagonal $\mathbb{C}^*\hookrightarrow G\left(\widetilde{d}\right)$ acts trivially on $\overline{R\left(\widetilde{d}\right)}$ and $G\left(\widetilde{d}\right)\big/\mathbb{C}^*\cong G(d)$.
Define \[\mathfrak{P}\left(\widetilde{d}\right):=\text{Spec}\left(\mathcal{O}_{\overline{R\left(\widetilde{d}\right)}}\left[\mathfrak{g}(d)^{\vee}[1];\partial\right]\right)\Big\slash G\left(d\right).\]
Using dimensional reduction, see Subsections \ref{dimred0} and \ref{Koszuls}, we have that:
\[\text{MF}^{\text{gr}}_T\left(\pi^{-1}\left(\mathcal{Y}\right), W\right)\cong D^b_T\left(\mathfrak{P}\left(\widetilde{d}\right)^{\text{ss}}\right).\]
The moment map equations for $i\in I$ are the relations for the Nakajima quiver varieties $N(f,d)$ and the moment map equation for $\infty$ is superfluous \cite[page 33]{d2}. This means that $\mathfrak{P}\left(\widetilde{d}\right)^{\text{ss}}\cong N(f,d)$. 

Assumption B is satisfied for $\left(\widetilde{Q^f}, \widetilde{Q}, \theta^f, 0\right)$. By Proposition \ref{repres} and by the discussion in Subsection \ref{compHA}, the algebra $\text{KHA}_T(Q)\cong \text{KHA}_T^{\text{gr}}\left(\widetilde{Q}, \widetilde{W}\right)$ acts on  $$\bigoplus_{d\in \mathbb{N}^I}
K_0^T\left(N(f,d)\right).$$


\subsection{} Let $J$ be the Jordan quiver. Its tripled quiver is $(Q_3,W_3)$. Consider $T\subset \left(\mathbb{C}^*\right)^3$ which fixes $W_3$. We obtain an action of $\text{KHA}_T(J)\cong \text{KHA}_T^{\text{gr}}\left(Q_3,W_3\right)$ on
\begin{equation}\label{nakajima}
\bigoplus_{d\geq 0} K^T_0\left(\text{Hilb}\left(\mathbb{A}^2_{\mathbb{C}},d\right)\right).
\end{equation}
The algebra $\text{KHA}_T(J)$ has a subalgebra isomorphic to  $U_{q, t}^{>}\left(\widehat{\widehat{\mathfrak{gl}_1}}\right)$ \cite{sv1}, \cite{ne1}. In \cite{ft}, \cite{sv1}, the authors construct a representations of the full quantum group $U_{q, t}\left(\widehat{\widehat{\mathfrak{gl}_1}}\right)$ on the vector space \eqref{nakajima}.

\section{The wall-crossing theorem}\label{4}
\subsection{}
Let $\theta\in \mathbb{Q}^I$ be a stability condition for $Q$ and define the slope $$\tau(d)=\frac{\sum_{i\in I} \theta^id^i}{\sum_{i\in I} d^i}.$$
For any partition $\dd=(d_1,\cdots, d_k)$ of $d$ with $d_i\in\mathbb{N}^I$, consider the maps:
\[\times_{i=1}^k\X(d_i)\xleftarrow{q_{\dd}}\X(d_1,\cdots, d_k)\xrightarrow{p_{\dd}}\X(d).\]
The stack $\X(d)$ has a Harder-Narasimhan stratification with strata 
$$p_{\dd}\left(q_{\dd}^{-1}\left(\times_{i=1}^k\X(d_i)^{\text{ss}}\right)\right)=:\X'(\underline{d})$$ 
corresponding to partitions $\dd=(d_1,\cdots,d_k)$ with $k\geq 2$, $d_i$ non-zero for $1\leq i\leq k$, and $\mu_1>\cdots>\mu_k$ where $\mu_i:=\tau(d_i)$.

Let $I$ be the set of such partitions. Consider two partitions $\underline{d}=(d_1,\cdots, d_k)$ and $\underline{e}=(e_1,\cdots, e_s)$ in $I$. We say that $\dd<\ee$ if $k>s$. Then $\X'(\underline{d})$ with $\underline{d}\in I$ are a stratification as in Subsection \ref{KNstrata}. 

\subsubsection{}
For $w\in\mathbb{Z}$, let $D^b(\X(d))_w$ be the subcategory of $D^b(\X(d))$ of complexes on which the diagonal cocharacter $\lambda:=z\cdot\text{Id}$ of $G(d)$ acts with weight $w$.
The category $D^b(\X(d))$ has an orthogonal decomposition in categories $D^b(\X(d))_w$ for $w\in\mathbb{Z}$. 

More generally, if $\X$ is a stack and $T$ is a torus acting trivially on $\X$, there is an orthogonal decomposition of $D^b_T(\X)$ in categories $D^b_T(\X)_\chi$, where $\chi$ is a weight of $T$ and $D^b_T(\X)_\chi$ is the subcategory of sheaves on $D^b_T(\X)$ on which $T$ acts with weight $\chi$.

\subsubsection{}
As a corollary of \cite[Theorem 2.10, Amplification 2.11]{hl}, see the semiorthogonal decomposition \eqref{hl}, we have that:

\begin{prop}\label{wallcros}
The category $D^b(\X(d))$ has a semiorthogonal decomposition with summands  $$p_{\dd*}q_{\dd}^*\left( D^b\left(\times_{i=1}^k\X(d_i)^{\text{ss}}\right)_\chi\right)\cong D^b\left(\times_{i=1}^k\X(d_i)^{\text{ss}}\right)_\chi,$$  where $\dd\in I$ and $\chi=(w_1,\cdots,w_k)$ is a weight of $(\mathbb{C}^*)^k$. The analogous decomposition holds for $D^b_{\mathbb{C}^*}(\X(d))$ where $\mathbb{C}^*$ is a subgroup of $\left(\mathbb{C}^*\right)^E$.
\end{prop}

Using Propositions \ref{prop1} and \ref{prop2}, we obtain the following corollary:

\begin{cor}\label{wccor}
(a) The category $\text{MF}\left(\X(d), W\right)$ has a semiorthogonal decomposition with summands  $$p_{\dd*}q_{\dd}^*\left( \text{MF}\left(\times_{i=1}^k\X^{\text{ss}}(d_i), W\right)_{\chi}\right)\cong \text{MF}\left(\times_{i=1}^k\X^{\text{ss}}(d_i), W\right)_{\chi},$$  where $\dd\in I$ and $\chi=(w_1,\cdots, w_k)$ is a weight of $\left(\mathbb{C}^*\right)^k$.

(b) Assume that $(Q,W)$ satisfies Assumption A and consider the corresponding $\text{MF}^{\text{gr}}$.
Then the category $\text{MF}^{\text{gr}}\left(\X(d), W\right)$ has a semiorthogonal decomposition with summands $$p_{\dd*}q_{\dd}^*\left( \text{MF}^{\text{gr}}\left(\times_{i=1}^k\X(d_i)^{\text{ss}}, W\right)_{\chi}
\right)\cong \text{MF}^{\text{gr}}\left(\times_{i=1}^k\X(d_i)^{\text{ss}}, W\right)_{\chi},$$  where $\dd\in I$ and $\chi=(w_1,\cdots, w_k)$ is a weight of $\left(\mathbb{C}^*\right)^k$.
\end{cor}

\begin{proof}
Let $\dd\in I$ and let $\chi$ be a weight as above.
Let $\textbf{C}$ be a subcategory of $D^b\left(\times_{i=1}^k\X(d_i)\right)_{\chi}$ on which $p_{\dd*}q_{\dd}^*$ is fully faithful. Then 
\[\text{MF}\left(p_{\dd*}q_{\dd}^*\,\textbf{C}, W_d\right)\cong p_{\dd*}\text{MF}\left(q^*_{\dd}\,\textbf{C}, p_{\dd}^*W_d\right)\cong p_{\dd*}q_{\dd}^*\,\text{MF}\left(\textbf{C}, \oplus_{i=1}^kW_{d_i}\right).\] The statement in part (a) follows from Proposition \ref{prop1}. The statement for $\text{MF}^{\text{gr}}$ follows similarly using Proposition \ref{prop2}.
\end{proof}


\subsubsection{}\label{kunn} We say that $(Q,W)$ satisfies \textbf{Assumption C} if for all $d,e\in\mathbb{N}^I$ and all stability conditions $\theta$, the Thom-Sebastiani maps are isomorphisms:
\begin{multline*}
    \text{TS}:K_0\big(\text{MF}(\X(d)^{\text{ss}}, W_d)\big)\otimes K_0\big(\text{MF}(\X(e)^{\text{ss}}, W_e)\big)\xrightarrow{\sim}\\ K_0\big(\text{MF}(\X(d)^{\text{ss}}\times\X(e)^{\text{ss}}, W_d\oplus W_e)\big).
    \end{multline*}
Under Assumption A, we can formulate the analogous assumption for $\text{MF}^{\text{gr}}$, which by \cite[Corollary 3.13]{T2} is the same as the Assumption C above. 
Any pair $(Q,0)$ satisfies Assumption C. In \cite{P2}, we check that any tripled quiver $\left(\widetilde{Q},\widetilde{W}\right)$ satisfies Assumption C.

\begin{thm}\label{wallcrossing}
(a) Assume that $(Q,W)$ satisfies Assumption C. Let $\theta$ be a stability condition.
There is an isomorphism of vector spaces
$$\text{KHA}(Q,W)\xrightarrow{\sim} \bigotimes_{\mu\in\mathbb{Q}} \text{KHA}(Q,W)_{\mu}.$$

(b) Assume that $(Q,W)$ satisfies Assumptions A and C and consider the corresponding $\text{KHA}^{\text{gr}}$. Let $\theta$ be a stability condition.
There is an isomorphism of vector spaces
$$\text{KHA}^{\text{gr}}(Q,W)\xrightarrow{\sim} \bigotimes_{\mu\in\mathbb{Q}} \text{KHA}^{\text{gr}}(Q,W)_{\mu}.$$
\end{thm}

The product on the right hand side on Theorem \ref{wallcrossing} is an ordered product taken after descending slopes, see also its analogue in cohomology \cite[Theorem B]{dm}.

\begin{proof}
Let $d\in\mathbb{N}^I$. We discuss the proof for part (a), part (b) follows from part (a). The statement follows from the isomorphism of vector spaces
\begin{equation}\label{decoKnot}
K_0\left(\text{MF}\left(\X(d), W\right)\right)\cong\bigoplus \bigotimes_{i=1}^k K_0\left(\text{MF}\left(\X(d_i)^{\text{ss}}, W\right)\right),
\end{equation}
where the sum is after all partitions $\underline{d}=(d_1,\cdots,d_k)$ in $I$. For any $d\in\mathbb{N}^I$, there are decompositions
\[K_0\left(\text{MF}\left(\X(d)^{\text{ss}}, W\right)\right)\cong\bigoplus_{w\in\mathbb{Z}}K_0\left(\text{MF}\left(\X(d)^{\text{ss}}, W\right)_w\right).\]
The decomposition \eqref{decoKnot} now follows from Corollary \ref{wccor} and the isomorphism from Assumption C.
\end{proof}

\subsection{Example}
Let $Q$ be a type A quiver with vertices labelled $1$ to $n$ and edges from $i$ to $i+1$ for $1\leq i\leq n-1$. 
\subsubsection{} Consider first the stability condition \[\theta: \theta^1<\cdots<\theta^n.\] Denote by $\varepsilon_i$
the dimension vector with $1$ in vertex $i$ and zero everywhere else. The $\theta$-semistable representations are at dimension vectors $n\varepsilon_i$ for $n$ a nonnegative integer. 
For $1\leq i\leq n$ a vertex, let $r_i$ be the unique representation of dimension $\varepsilon_i$. Then $$\X(n\varepsilon_i)^{\text{ss}}=\left(r_i^{\oplus n}\right)/ GL(n),$$ so we have that $$S_i:=\bigoplus_{n\geq 0} K_0\left(\X(n\varepsilon_i)^{\text{ss}}\right)\cong\bigoplus_{n\geq 0}K_0(BGL(n)).$$
As a corollary of Theorem \ref{wallcrossing}, we obtain that:
\begin{cor}\label{wccor2} $\text{KHA}(Q, 0)$ is generated by the $\varepsilon_i$-graded pieces with $1\leq i\leq n$ dimensional pieces under the multiplication map: 
$$\text{KHA}(Q, 0)\cong \bigotimes_{i=1}^n S_i.$$ 
\end{cor}

\subsubsection{}
Choose next the stability condition \[\theta': \theta^1>\cdots>\theta^n.\] The semistable representations are at the multiples of the roots $r_1,\cdots, r_N$ of the Lie algebra associated to $Q$. Consider an algebra $S_i\cong \bigoplus_{n\geq 0}K_0(BGL(n))$ for any $1\leq i\leq N$. An analysis as above and Theorem \ref{wallcrossing} imply that:
\begin{cor}\label{wccor3}
The $\text{KHA}(Q,0)$ is generated by the $r_i$-graded pieces with $1\leq i\leq N$ under the multiplication map: $$\text{KHA}(Q,0)\cong \bigotimes_{i=1}^N S_i.$$ 
\end{cor}

The analogues of Corollaries \ref{wccor2} and \ref{wccor3} for CoHA were proved by Rimanyi \cite{r}. The case of $A_2$ in cohomology has been treated by Kontsevich--Soibelman \cite[Section 2.8]{ks}.

\end{document}